\documentclass[12pt]{amsart}
\usepackage{amssymb,amsmath}
\usepackage{geometry}    
\usepackage{mathrsfs}

\usepackage{graphicx}

\usepackage{vmargin}

\usepackage{tikz}  

\setmargrb{1in}{1in}{1in}{1in}

\newtheorem{theorem}{Theorem}[section]

\newtheorem{proposition}{Proposition}

\theoremstyle{definition}
\newtheorem{definition}{Definition}

\newtheorem{problem}{Problem}

\newcommand{\ux}{\underline{\xi}}
\newcommand{\uv}{\underline{\varphi}}
\newcommand{\ov}{\overline{\varphi}}
\newcommand{\om}{\omega}
\newcommand{\wo}{\widehat{\omega}}
\newcommand{\os}{\omega^{\ast}}
\newcommand{\wos}{\widehat{\omega}^{\ast}}

\begin{document}

\title[Optimality of two inequalities]{Optimality of two inequalities for exponents of Diophantine approximation}

\author{Johannes Schleischitz}

\thanks{Middle East Technical University, Northern Cyprus Campus, Kalkanli, G\"uzelyurt \\
	johannes@metu.edu.tr ; jschleischitz@outlook.com}

\dedicatory{Dedicated to the 50th birthday of Yann Bugeaud}

\begin{abstract}
	We investigate two inequalities of Bugeaud and Laurent, each involving 
	triples of classical exponents of Diophantine approximation
	associated to $\ux\in\mathbb{R}^n$.
	We provide a complete description of parameter triples 
	that admit equality for suitable $\ux$, which turns
	out rather surprising. 
	For $n=2$ our results agree with work
	of Laurent.
	Moreover, we establish lower bounds for the Hausdorff and packing dimensions
	of the involved $\ux$, and in special cases we can show they are sharp. 
	Proofs are based on the variational principle in parametric geometry of
	numbers, we enclose sketches of associated combined graphs (templates)
	where equality is feasible. A twist of our construction provides 
	refined information 
	on the joint spectrum of the respective exponent triples. 
\end{abstract}

\maketitle



{\footnotesize{

Math Subject Classification 2010: 11J13, 11J82}}

\vspace{1mm}

\section{Introduction}

\subsection{Classical exponents of approximation} \label{intro}

Let $n\geq 1$ be an integer and $\ux=(\xi_{1},\ldots,\xi_{n})\in\mathbb{R}^n$
with $\{1,\xi_{1},\ldots,\xi_{n}\}$ linearly independent over $\mathbb{Q}$.
Let the (possibly infinite) exponents of approximation 
$\om$ and $\wo$ resp. be defined as the suprema of reals $u$ so that
the system
\begin{equation}  \label{eq:om}
1\leq x\leq X, \qquad \max_{1\leq i\leq n} | x\xi_{i}-y_{i}|\leq X^{-u}
\end{equation}
has a solution in integer vectors $(x,y_{1},\ldots,y_{n})$
for arbitrarily large and all large $X$, respectively.
 Let $\os, \wos$ be the supremum of $v$ so that
 \[
 1\leq \max_{1\leq i\leq n} |a_{i}|\leq X, \qquad  | a_{0}+a_{1}\xi_{1}+\cdots+a_{n}\xi_{n}|\leq X^{-v}
 \]
has a solution in integers $a_{i}$ 
for arbitrarily large $X$ and all large $X$, respectively.  
By variants of Dirichlet's Theorem, for any $\ux\in\mathbb{R}^n$ we have
\begin{equation}  \label{eq:diri}
\infty\geq \om\geq \wo\geq \frac{1}{n}, \qquad\qquad   \infty\geq \os\geq \wos\geq n. 
\end{equation}

\subsection{Two inequalities by Bugeaud and Laurent}  \label{se}

Let $n\geq 2$ be an integer.
Bugeaud and Laurent~\cite{bulau} established that every $\ux$
as above satisfies the estimates
\begin{equation}  \label{eq:1}  \tag{BL1}
\om \geq   \frac{ (\wos -1) \os } { ((n-2)\wos+1)\os+(n-1)\wos  } 
\end{equation}
and 
\begin{equation}  \label{eq:2}  \tag{BL2}
\os \geq \frac{  (n-1)\om + \wo +n -2 }{1-\wo}.
\end{equation}
In view of \eqref{eq:diri} they imply Khintchine's 
transference inequalities~\cite{khint}
\begin{equation}  \label{eq:khin}
\os\geq n\om+n-1, \qquad \om\geq \frac{\os}{(n-1)\os+n}.
\end{equation}
These are known to be sharp for any parameter pairs 
induced by $\om\in [1/n,\infty]$
resp. $\os\in[n,\infty]$.
Moreover, combining \eqref{eq:1}, \eqref{eq:2} implies 
non-trivial relations between $\om$ and $\wo$, and likewise
between $\os$ and $\wos$.
For $n=2$, they become
\begin{equation}  \label{eq:oja}
\om\geq \frac{\wo^{2}}{1-\wo},  \qquad\qquad \os\geq \widehat{\omega}^{\ast 2}-\wos, 
\end{equation}
already known by Jarn\'ik~\cite{jacj} and
are again sharp for any parameter pairs 
induced by $\om\in [1/n,\infty]$
resp. $\os\in[n,\infty]$.
For $n>2$, the implied relations turn out to be no longer best possible. Marnat and Moshchevitin~\cite{mamo} settled the sharp estimates
conjectured by 
Schmidt and Summerer~\cite{ssmj} along their discussion
of the ''regular graph'', who gave proofs
for $n=3$
themselves. See also Schmidt and Summerer~\cite{ssmh},  Moshchevitin~\cite{gemo} and Rivard-Cooke's 
thesis~\cite{rivart}.

As noticed by German and Moshchevitin~\cite{germosh}, the inequalities
\eqref{eq:1} resp. \eqref{eq:2} split into pairs of inequalities
respectively given by
\begin{equation}  \label{eq:ggmm}
\frac{1+\omega^{\ast -1} }{ 1+\omega^{-1} } \geq \wo \geq \frac{1-\widehat{\omega}^{\ast -1} }{n-1},
\qquad \frac{1+\os}{ 1+\om } \geq \wos \geq \frac{n-1}{1-\wo}.
\end{equation}
The respective left inequalities originate in 
Schmidt and Summerer~\cite{ss}, the right ones in German~\cite{german},
with alternative proofs given later 
in~\cite{ognm}, \cite{germosh} resp. \cite{ssa2}.
In particular, in case equality in \eqref{eq:1} or \eqref{eq:2}, both respective splitting inequalities
must be identities as well. Thus our results below also establish
optimality of these splitting inequalities for certain parameter ranges 
(shown for the resp. right inequalities already in~\cite{ssa2}).

\section{Main result: Description of equality cases}  \label{se2}

For $n=2$ all estimates \eqref{eq:1}, \eqref{eq:2}
are sharp by a result of Laurent~\cite{laurent},
as already pointed out in~\cite{bulau}. 
For general $n\geq 2$, when $\os\geq n, \wos=n$ resp. $\om\geq 1/n, \wo=1/n$ inequalities \eqref{eq:1} resp. \eqref{eq:2} simplify to \eqref{eq:khin}
and are sharp, as noticed in Section~\ref{se}.
Other than that,
when $n>2$ the optimality
of these estimates was very open. The main purpose of this note is to give a 
comprehensive description of equality cases. 
 
Let us start with \eqref{eq:1}.
We explicitly
determine the submanifold generated by the intersection of the spectrum
of $(\os, \wos, \om)\subseteq (\mathbb{R}\cup \{\infty\})^3$ induced by $\ux\in\mathbb{R}^n$ 
with the hypersurface in $(\mathbb{R}\cup \{\infty\})^3$ induced by equality in \eqref{eq:1}.
For given $n\geq 2$ and a real parameter $x\geq n$, define
\[
\rho_{1}(n,x)= \frac{x}{(n-1)x+n},
\]
and
\[
\rho_{2}(n,x)= \frac{(2n-4)x^{2} + (2n-1- \sqrt{(4n-4)x+1} )x - \sqrt{(4n-4)x+1}+1}
{2((n-2)^2x^{2}+(2n^2-6n+3)x+n^2-2n)}.
\]
By taking limits we put $\rho_{1}(n,\infty)=1/(n-1)$ and $\rho_{2}(n,\infty)=1/(n-2)$, where we consider $1/0=+\infty$. 
Observe $\rho_{1}(n,n)= \rho_{2}(n,n)=1/n$.
Our results will involve
Hausdorff and packing dimension, see~\cite{falconer} for an introduction.

\begin{theorem}  \label{t0}
	Let $n\geq 2$ be an integer. Then precisely for triples $(w^{\ast},\widehat{w}^{\ast},w)\subseteq (\mathbb{R}\cup\{\infty\})^3$ 
	that can be parametrized as
	\begin{equation}  \label{eq:ride}
	w^{\ast} \in [n,\infty], \qquad  w\in[\rho_{1}(n,w^{\ast}), \rho_{2}(n,w^{\ast}) ], \qquad 
	\widehat{w}^{\ast}= \frac{w^{\ast}(w+1)  }{ w^{\ast}-(n-2)ww^{\ast}-(n-1)w }
	\end{equation}
	there is $\ux\in \mathbb{R}^n$ which induces 
	equality in \eqref{eq:1} and
	\begin{equation} \label{eq:dieequ}
	\os=w^{\ast}, \qquad \wos=\widehat{w}^{\ast}, \qquad \om=w.
	\end{equation}
	In fact, 
	for each admissible parameter 
	triple $(w^{\ast},\widehat{w}^{\ast},w)$ as in \eqref{eq:ride}, the
	set of $\ux$ inducing \eqref{eq:dieequ} has Hausdorff dimension  
	at least $n-2$ and packing dimension at least $n-2+1/n$.
\end{theorem}

The identity in \eqref{eq:ride} is a reformulation of 
equality in \eqref{eq:1} and its explicit statement is purely conventional. 
The lower bound 
$\rho_{1}$ reflects identity in the right estimate of \eqref{eq:khin}, so
it is as small as it can possibly be. 
On the other hand, the surprising bound $\rho_{2}$
is strictly smaller than the bound $(w^{\ast}-n+1)/n$ from the left estimate 
in \eqref{eq:khin}, unless in trivial cases. 
Thus, by the optimality of Khintchine's estimates,
when $w^{\ast}\in (n,\infty]$ and
$w\in (\rho_{2}, (w^{\ast}-n+1)/n]$, there are $\ux$ whose associated exponents
satisfy $\os=w^{\ast}, \om=w$ but there is no $\ux$
with additional equality in \eqref{eq:1}. 
We remark that if $w=\rho_{2}$, then
as $w^{\ast}\to \infty$ also $\widehat{w}^{\ast}\to\infty$, and 
$w\to 1/(n-2)$. 
We point out that for $\ux$ as in the theorem inducing \eqref{eq:dieequ},
the remaining exponent $\wo$ can be determined as
\begin{equation}  \label{eq:istg}
\wo=\frac{1+w^{\ast -1} }{ 1+w^{-1} }
\end{equation}
by \eqref{eq:ggmm}. 
Consequently, if $w=\rho_{2}$ then $\wo\to 1/(n-1)$ as $w^{\ast}\to\infty$. 
As $w^{\ast}\to n$ from above, both $\rho_{1}, \rho_{2}$ tend to $1/n$
as it needs to be. For refinements of the metrical claim see
Section~\ref{ela} below.

For $n=2$ the claim agrees with the findings of Laurent~\cite{laurent}
discussed above.
For $n=2$ and $\om=\rho_{2}$
we have equality in the inequalities of \eqref{eq:oja}, thereby
we obtain a regular graph as mentioned in Section~\ref{se}.
For $n>2$, no regular graph can appear upon identity in \eqref{eq:1},
unless in the trivial case $w^{\ast}=\widehat{w}^{\ast}=n, w=1/n$.

Our results suggest that solutions to the following questions 
are in reach.

\begin{problem}  \label{problem1}
	For given $\os$ and $\om>\rho_{2}(n,\os)$ 
	find sharp estimates improving \eqref{eq:1}. 
	Ideally, determine the spectrum of
	$(\os, \wos, \om)\subseteq (\mathbb{R}\cup \{\infty\})^3$. 
\end{problem}

We turn towards
the dual estimates \eqref{eq:2}. For $n\geq 2$ and $x\geq 1/n$, define
\[
\tau_{1}(n,x)= nx+n-1,
\]
and
\[
\tau_{2}(n,x)= \frac{x^2}{2} +\left( n-\frac{1}{2}+ \frac{ \sqrt{x (x+4n-4)} }{2} \right) x
+ \frac{ \sqrt{x (x+4n-4)} }{2} + n-2.
\]
By taking limits we extend it to $\tau_{j}(n,\infty)=\infty$ for $j=1,2$. 
Our result reads as follows.

\begin{theorem} \label{t2}
	Let $n\geq 2$ be an integer. Then precisely for triples $(w,\widehat{w},w^{\ast})\subseteq (\mathbb{R}\cup\{\infty\})^3$ 
	that can be parametrized by the properties
	\begin{equation}  \label{eq:ride0}
	w\in [1/n,\infty], \qquad  w^{\ast}\in[\tau_{1}(n,w), \tau_{2}(n,w) ], \qquad 
	\widehat{w}= \frac{w^{\ast}-(n-1)w-n+2  }{ 1+w^{\ast}}
	\end{equation}
	there is $\ux\in \mathbb{R}^n$ which induces 
	equality in \eqref{eq:2} and
	\begin{equation} \label{eq:dieequ0}
	\om=w, \qquad \wo=\widehat{w}, \qquad \os=w^{\ast}.
	\end{equation}
	For each suitable parameter triple $(w,\widehat{w},w^{\ast})$, 
	the set of
	associated $\ux$ inducing \eqref{eq:dieequ0} has packing dimension at least $1/2$, and positive Hausdorff dimension
	if $w<\infty$.
\end{theorem}

Analogous remarks as for Theorem~\ref{t0} apply. The lower bound
$\tau_{1}$ reflects equality in the left estimate 
of \eqref{eq:khin}, whereas $\tau_{2}$ is strictly 
smaller than the value induced by equality in the right inequality of \eqref{eq:khin}, 
unless if $\om=\os=n$.  
Let $w^{\ast}=\tau_{2}$. Then, 
as $w\to \infty$ we have $\widehat{w}\to 1$ and $w^{\ast}\to \infty$, where
the latter agrees with the obvious estimate $\os\geq \om$. Note that $1$ is the largest value $\wo$ can attain. 
For $\ux$ as in Theorem~\ref{t2}, the missing
exponent $\wos$ can again be evaluated as
\begin{equation} \label{eq:istg2}
\wos=\frac{1+w^{\ast}}{ 1+w }
\end{equation}
by \eqref{eq:ggmm}.
If $\os=\tau_{2}(n,\om)$ then
again as $\om\to \infty$ also $\wos\to\infty$, and iff $n=2$ this leads to the regular graph.
We formulate the analogous questions to Problem~\ref{problem1}.
	
	\begin{problem}  \label{problem2}
		For $\os>\tau_{2}(n,\om)$, 
		find sharp estimates improving \eqref{eq:2}.
		Ideally, determine the spectrum of
		$(\om, \wo, \os)\subseteq (\mathbb{R}\cup \{\infty\})^3$.
	\end{problem}

As partial results to Problems~\ref{problem1},~\ref{problem2},
for any triples satisfying
\begin{equation} \label{eq:A1}
w^{\ast}\in [n,\infty], \quad w\in [\rho_{1}(n,w^{\ast}), \rho_{2}(n,w^{\ast})], \quad \widehat{w}^{\ast}\in \left[n,\frac{w^{\ast}(w +1)}{w^{\ast}-(n-2)w w^{\ast} -(n-1)w}\right]
\end{equation}
resp.
\begin{equation}   \label{eq:A2}
w \in [1/n,\infty], \quad w^{\ast}\in [\tau_{1}(n,w),\tau_{2}(n,w)] , \quad
\widehat{w}\in \left[\frac{1}{n}, \frac{w^{\ast}-(n-1)w-n+2}{1+w^{\ast}}\right]
\end{equation}
for suitable $\ux\in \mathbb{R}^n$ we still get \eqref{eq:dieequ} resp. \eqref{eq:dieequ0}. Clearly the ranges for $\wos, \wo$ are optimal.
See Section~\ref{ee} for the proof. Problems~\ref{problem1},~\ref{problem2} can be considered partial problems
towards finding the entire spectrum in $\mathbb{R}^{2n+2}$ of all
extremal values of successive minima exponents, which is wide open for $n>2$.
See Section~\ref{ppp} for details. 

We outline the rest of the paper. In Section~\ref{pgnu} we 
introduce parametric geometry of numbers in the notion of Schmidt and Summerer~\cite{ssmh} and formulate a special case the variational principle by Das, Fishman, Simmons, Urba\'nski~\cite{dfsu1},~\cite{dfsu2} within this framework. We append relations to notions of Roy~\cite{royann},~\cite{royspec} and Schmidt, Summerer~\cite{ssma}.
In Section~\ref{se4} we reformulate 
Theorems~\ref{t0},~\ref{t2} into this language, and
refine the metrical claims. In
Sections~\ref{se5},~\ref{se6} these claims and 
\eqref{eq:A1}, \eqref{eq:A2} are proved using the prerequisites from~\cite{dfsu1},~\cite{dfsu2}.
It is worth noting that everything except from
the metrical claims can be alternatively obtained from
Roy~\cite{royann},~\cite{royspec} in place of~\cite{dfsu1},~\cite{dfsu2}, see Section~\ref{roy} for details. Finally
we close with some remarks interconnecting our work with~\cite{bulau},~\cite{royspec} in Section~\ref{comments}.

\section{Parametric geometry of numbers}  \label{pgnu}

\subsection{Parametric functions and their extremal values} \label{ppp}

Let us interpret the simultaneous approximation problem
\eqref{eq:om} as a parametric successive minima 
problem. For $q>0$ a parameter,
let $K(q)\subseteq \mathbb{R}^{n+1}$ be the box of points 
$(z_{0},z_{1},\ldots, z_{n})$ that satisfy
\[
\vert z_{0}\vert \leq e^{nq},\qquad \max_{1\leq i\leq n} \vert z_{i}\vert \leq e^{-q}.
\] 
Further let $\Lambda=\Lambda_{\underline{\xi}}$ be the lattice consisting of all points
of the form
$\{ (x,\xi_{1}x-y_{1}, \ldots, \xi_{n}x-y_{n}): x,y_{i}\in\mathbb{Z}\}$. 
Denote by  $\lambda_{1}(q),\ldots,\lambda_{n+1}(q)$
the successive minima
of $K(q)$ with respect to
$\Lambda$, to obtain functions of $q$. 
For $1\leq j\leq n+1$, derive $\varphi_{j}(q)=\log \lambda_{j}(q)/q$
and define the lower and upper limits
\[
\uv_{j}= \liminf_{q\to\infty} \varphi_{j}(q),
\qquad  \ov_{j}= \limsup_{q\to\infty} \varphi_{j}(q).
\]
These quantities lie within the interval $[-n,1]$.
Schmidt and Summerer~\cite[(1.8), (1.9)]{ssmh} observed
they are connected to exponents of
Section~\ref{intro} via the transference identities
\begin{equation}  \label{eq:a}  \tag{T1}
(1+\om)(n+\uv_{1})=(1+\wo)(n+\ov_{1})=n+1,
\end{equation}
and
\begin{equation} \label{eq:b}   \tag{T2}
(1+\os)(1-\ov_{n+1})=(1+\wos)(1-\uv_{n+1})=n+1.
\end{equation}
Hereby we mean $\om=\infty$ iff $\uv_{1}=-n$ and likewise for
other identities. See~\cite[Corollary~8.5]{og} for a generalization.
In terms of $\ov_{j}, \uv_{j}$ 
Khintchine's estimates \eqref{eq:khin} simply read
\begin{equation}  \label{eq:khin2}
\ov_{n+1} \geq -\frac{ \uv_{1} }{n}, \qquad \uv_{1} \leq 
-\frac{ \ov_{n+1} }{n}.
\end{equation} 
The deduction from \eqref{eq:a}, \eqref{eq:b}
is carried out in Remark~(b) in~\cite{ss}.
Equivalent formulations of \eqref{eq:oja}
obtained
similarly can be found in
~\cite[(1.20), (1.20$^{\prime}$)]{ssmh}. 
%
%
Moreover \eqref{eq:a}, \eqref{eq:b}
imply that \eqref{eq:1}, \eqref{eq:2} are respectively 
equivalent to
\begin{equation}  \label{eq:3}  \tag{SS1}
n\uv_{1} +\ov_{n+1} \leq -\uv_{n+1}\cdot \left(\frac{n+1}{n-1}+\uv_{1}+\frac{2}{n-1}\ov_{n+1} \right)
\end{equation}
and
\begin{equation} \label{eq:4}   \tag{SS2}
n\ov_{n+1} +\uv_{1} \geq -\ov_{1}\cdot \left(\frac{n+1}{n-1}+\ov_{n+1}+\frac{2}{n-1}\uv_{1} \right).
\end{equation}
This was again already observed by Schmidt and Summerer~\cite{ssmh} who 
provided independent proofs of \eqref{eq:3}, \eqref{eq:4} based on
parametric geometry of numbers, and thereby
new proofs \eqref{eq:1}, \eqref{eq:2}.
Other proofs of \eqref{eq:1}, \eqref{eq:2} that again rely on parametric geometry of numbers came as a byproduct in~\cite{ichnyj}, see~\cite[(23)]{ichnyj} and~\cite[Remark~6]{ichnyj}.
The latter proofs from~\cite{ichnyj} give some information
on $(\varphi_{1}(q), \ldots, \varphi_{n+1}(q))$ as a function 
$[0,\infty)\to [-n,1]^{n+1}$, i.e. on the combined graph
defined in Section~\ref{roy} below,
in case of equality. This observation inspired this note.

In recent years, much work has been done on the joint spectrum 
of exponents, that is the subset of $\mathbb{R}^{2n+2}$ that occurs as
\[
(\uv_{1},\ldots,\uv_{n+1},\ov_{1},\ldots,\ov_{n+1})\in\mathbb{R}^{2n+2}
\]
when $\ux\in\mathbb{R}^n$ runs through all vectors that are 
$\mathbb{Q}$-linearly independent
together with $\{1\}$. See for example~\cite{rivart},~\cite{rivroy},~\cite{royspec},~\cite{roytop},~\cite{ssmh},~\cite{ssa2},~\cite{ssma}, in particular the detailed exposition in Section 1 of~\cite{roytop}. In view of the equivalent claims \eqref{eq:3}, \eqref{eq:4},
our claims in Section~\ref{se2} contribute to this
area by providing new results on the projection to 
the three-dimensional spaces with coordinate
variables $(\uv_{1},\uv_{n+1},\ov_{n+1})$
and $(\uv_{1},\ov_{1},\ov_{n+1})$, respectively.
See Theorems~\ref{t1},~\ref{t3} below. We remark that for $n=2$, 
a complete description of the joint spectrum via a system of inequalities was
given in~\cite[Theorem~11.5]{roytop}, thereby containing implicitly the work of Laurent~\cite{laurent}, Schmidt, Summerer~\cite{ssma} as well as the case $n=2$ of our Theorems~\ref{t0},~\ref{t2}. However, the explicit deductions from~\cite{roytop} seem cumbersome.

\subsection{$n$-templates and the variational principle} \label{roy}

From the functions $\varphi_{j}(q)$ and $\lambda_{j}(q)$ 
associated to $\ux\in\mathbb{R}^n$ as
defined in Section~\ref{ppp}, we
derive $L_{j}(q)= q\varphi_{j}(q)= \log \lambda_{j}(q)$.
These functions are piecewise linear with slopes among $\{-n,1\}$, and by Minkowski's Second Convex Body Theorem their sum is uniformly bounded
\begin{equation}  \label{eq:bsum}
\left|\sum_{j=1}^{n+1} L_{j}(q)\right| \leq C(n), \qquad\qquad q\in[0,\infty).
\end{equation}
The values $\uv_{j}, \ov_{j}$ are just the 
extremal average slopes
of the $L_{j}$ in a start segment, i.e.
\begin{equation}  \label{eq:idf}
\uv_{j}= \liminf_{q\to\infty} \frac{L_{j}(q)}{q} , \qquad
\ov_{j}= \limsup_{q\to\infty} \frac{L_{j}(q)}{q}, \qquad\qquad (1\leq j\leq n+1).
\end{equation}

We call $\textbf{L}_{\ux}(q)=(L_{1}(q),\ldots,L_{n+1}(q))$ 
on $q\in[0,\infty)$ the combined graph associated to $\ux\in\mathbb{R}^n$.
We approximate it
by easier systems $\textbf{P}=(P_{1}(q),\ldots,P_{n+1}(q))$ without error term as in \eqref{eq:bsum}, but where we may glue consecutive functions $P_{j}$ on intervals where they differ by a small amount. 
Let 
\[
Z(j) = \left\{ j , j-1-n \right\}, \quad P_{0}(q)=-\infty, \quad P_{n+2}(q)=+\infty.
\]
Then following~\cite[Definition~5.1]{dfsu2}, an 
elegant formal description of the family of functions
we consider can be stated as follows. 

\begin{definition} \label{defnt}
	We call a continuous, piecewise-linear map 
	$\textbf{P}: [0,\infty)\to \mathbb{R}^{n+1}$ 
	an $n$-template if the component functions $P_{j}(q)$ 
	satisy $(i)-(iv)$ below.
	\begin{enumerate}
		\item[$(i$)] We have	
		\begin{equation} \label{eq:sumv}
		\sum_{j=1}^{n+1} P_{j}(q) = 0, \qquad\qquad q\in [0,\infty).
		\end{equation}
		\item[$(ii$)] $P_{1}(q)\leq P_{2}(q)\leq \cdots \leq P_{n+1}(q)$ for all $q\in[0,\infty)$
		\item[$(iii$)] For any $q\in[0,\infty)$ and $1\leq j\leq n+1$, if $P_{j}$ is 
		differentiable at $q$ then
	    \[
	    -n \leq P_{j}^{\prime}(q) \leq 1.
	    \]    
	    	\item[$(iv)$] For $j=0,1,\ldots,n+1$ and every interval where $P_{j}<P_{j+1}$ holds, the function
	    \[
	    F_{j}(q) = P_{1}(q)+P_{2}(q)+\cdots+P_{j}(q)
	    \]
	    is convex and has slopes $F_{j}^{\prime}$ in $Z(j)$.	
	\end{enumerate} 
\end{definition}

In place of $(i)$ it would suffice to demand $P_{1}(0)=0$.
Indeed $\textbf{P}(q)$ is an $n$-template iff $\textbf{P}(q)/n$
is a balanced $n\times 1$-template in~\cite{dfsu1},~\cite{dfsu2}. 
Each $P_{j}$ is one-sided differentiable on $q\in [0,\infty)$ 
with slopes within the finite set $\{1\}\cup \{-k/(n+1-k): 0\leq k\leq n\}$.
We further remark that generalizing $3$-systems defined by Schmidt, Summerer~\cite{ssma} naturally to $(n+1)$-systems for $n\geq 1$, 
leads to special
cases of $n$-templates, with slopes of 
the $P_{j}$ restricted to $\{-n,1\}$ as for $L_{j}$.
More precisely, equipping the space of functions 
$[0,\infty)\to \mathbb{R}^{n+1}$ with 
the supremum norm,
the closure of the set of $(n+1)$-systems in~\cite{ssma} becomes 
our set of $n$-templates.  
By $(i)$, the sum of the slopes at points of differentiability
vanishes as well. In the special case of $(n+1)$-systems 
this means one component decays with slope $-n$ while the remaining
$P_{j}$ rise with slope $+1$. 
We further want to notice that upon applying some affine map, 
our $n$-templates correspond to the generalized $(n+1)$-systems
defined by Roy~\cite[Definition~5.1]{royspec}. So an equivalent
formulation of $n$-templates can be derived from a twist of~\cite[Definition~4.1]{royspec}.
This also implies that
the results in~\cite{royann},~\cite{royspec},~\cite{roytop}, remain
basically valid and we will use them occasionally below.

As noticed in~\cite{dfsu2},
the following special case of~\cite[Theorem~5.2]{dfsu2}
is already covered by
Schmidt and Summerer~\cite{ssmh} and Roy~\cite[Corollary 4.7]{royspec}.

\begin{theorem}  \label{dop}
	For any $\ux\in\mathbb{R}^n$, its associated combined graph
	has uniformly bounded distance from some suitable $n$-template $\textbf{P}$, i.e.
	\begin{equation}  \label{eq:jaroy}
	\|\textbf{L}_{\ux}(q)-\textbf{P}(q)\|= \max_{1\leq j\leq n+1} |P_{j}(q)-L_{j}(q)|\leq C(n), \qquad\qquad q\in[0,\infty).
	\end{equation}
	Conversely, for any 
	$n$-template
	$\textbf{P}$, there is $\ux\in\mathbb{R}^n$ 
	inducing \eqref{eq:jaroy} for some effective $C(n)$.
\end{theorem}

 In particular, if $\textbf{P}$ is as in the theorem for $\ux$, then
by \eqref{eq:idf} we see
\begin{equation}  \label{eq:landp}
\limsup_{q\to\infty} \frac{P_{j}(q)}{q}= \overline{\varphi}_{j}, \qquad
\liminf_{q\to\infty} \frac{P_{j}(q)}{q}= \underline{\varphi}_{j}.
\end{equation}

Theorem~\ref{deep} below refines the latter claim
of Theorem~\ref{dop} by adding metrical information.
 For $\mathbf{P}$ an $n$-template,
define its
local contraction rate at $q\in [q_{0},\infty]$ by
\[
\delta(\textbf{P}, q) = \kappa-1, \qquad\qquad
\kappa:= \{\max j: P_{j}^{\prime}(q) < 1\}.
\]
This agrees with the definition of $\delta(\textbf{f},I)$ in~\cite[(5.10)]{dfsu2} for $n\times 1$ templates, i.e. our $n$-templates.
See also~\cite{marnat}.
Derive the average contraction rate in the interval $[q_{0},q]$ by
\[
\Delta(\textbf{P},q)= \frac{1}{q-q_{0}} \cdot \int_{q_{0}}^{q}  \delta(\textbf{P}, u)\; du
\]
and the upper and lower contraction rates by
\[
\underline{\delta}(\textbf{P}) = \liminf_{q\to\infty} \Delta(\textbf{P}, q),
\qquad \overline{\delta}(\textbf{P}) = \limsup_{q\to\infty} \Delta(\textbf{P}, q).
\]
Denote by $\dim_{H}$ resp. $\dim_{P}$
the Hausdorff resp. packing dimensions, and call a family $\mathscr{F}$ 
of templates closed under finite perturbation if when
$\textbf{P}, \textbf{Q}$ are $n$-templates and
$\textbf{P}\in \mathscr{F}$
and $\Vert \textbf{P}-\textbf{Q}\Vert< \infty$, then also $\textbf{Q}\in \mathscr{F}$. Then
a consequence of the variational principle~\cite[Theorem~5.3]{dfsu2}
can be stated as follows.

\begin{theorem}  \label{deep}
	Let $\mathscr{F}$ be a family of $n$-templates 
	closed under finite perturbation.
	Denote by $\mathscr{M}=\mathscr{M}(\mathscr{F})$ the set
	of all $\ux\in\mathbb{R}^n$ whose associated
	combined graphs $\mathbf{L}_{\ux}(q)$ satisfy $\|\textbf{L}_{\ux}(q)-\textbf{P}(q)\|<\infty$
	for some (thus all) $\textbf{P}\in \mathscr{F}$. Then
	\[
	\dim_{H}(\mathscr{M}) = \sup_{ \mathbf{P}\in\mathscr{F} } \underline{\delta}(\textbf{P}), \qquad
	\dim_{P}(\mathscr{M}) = \sup_{ \mathbf{P}\in\mathscr{F} }\overline{\delta}(\textbf{P}).
	\]
\end{theorem}

In particular, for any given $n$-template
$\textbf{P}$, taking $\mathscr{F}$
the ''finite perturbation hull'' of $\{\textbf{P}\}$ 
and deriving $\mathscr{M}(\mathscr{F})= \mathscr{M}(\{\textbf{P}\})$ as above, 
we see
\[
\dim_{H}(\mathscr{M}) \geq \underline{\delta}(\textbf{P}), \qquad
\dim_{P}(\mathscr{M}) \geq  \overline{\delta}(\textbf{P}).
\]
As remarked in~\cite{dfsu1}, there is no equality in general
as $\delta$ is sensitive to perturbations. 

\section{Reformulating claims in terms of $\ov_{j},\uv_{j}$} \label{se4}

\subsection{Equivalent formulations of Theorems~\ref{t0},~\ref{t2}}  \label{se41}

For $n\geq 2$ an integer and $x\in [-n,1]$ write 
\begin{equation}  \label{eq:hdef}
g_{n}(x)= \frac{(3-2n)x+1-2n+\sqrt{(1-x)[(4n-5)x+4n^2-4n+1]} }{2(n-1)^2}.
\end{equation}
An equivalent formulation of Theorem~\ref{t0} in terms of the
quantities $\ov_{j}, \uv_{j}$ that we will prefer to
prove reads as follows.

\begin{theorem}  \label{t1}
	Let $n\geq 2$ be an integer. 
	Then precisely for triples $(t, \mu, \sigma)$ satisfying
	\begin{equation}  \label{eq:obey}
	t\in [0,1], \qquad \mu\in [g_{n}(t), -\frac{t}{n}], \qquad 
	\sigma= (1-n)\frac{t+n\mu }{n+1+2t+(n-1)\mu}
	\end{equation}
	there is $\ux\in\mathbb{R}^n$ inducing equality in \eqref{eq:3} and
	with associated values
	\[
	\ov_{n+1}= t, \qquad \uv_{n+1}=\sigma, \qquad \uv_{1}=\mu.
	\]
	For given $(t, \mu, \sigma)$ obeying the restrictions \eqref{eq:obey},
	the set of $\ux$ with these properties
	\begin{equation}  \label{eq:dim}
	\Theta=\Theta_{t,\mu}^{(n)}=\{ \ux\in\mathbb{R}^n:\; \ov_{n+1}= t, \quad \uv_{n+1}=\sigma, \quad \uv_{1}=\mu \}
	\end{equation}
	in fact has Hausdorff dimension at least $n-2$ and packing dimension at least $n-2+1/n$.
\end{theorem}

Write $\mu_{0}:= g_{n}(t)$ in the sequel.
The deduction of Theorem~\ref{t0} from Theorem~\ref{t1} relies
purely on \eqref{eq:a}, \eqref{eq:b}.
By equivalence of \eqref{eq:khin} and \eqref{eq:khin2} the upper bound $-t/n$ is a consequence of the bound $\rho_{1}$ and again the best we can hope for.
For $\rho_{2}$, we
further use the defining equation 
\begin{equation}  \label{eq:muiod}
(1+t+(n-1)\mu_{0})^2 - (1-t)(1-\mu_{0}) = 0
\end{equation}
of $\mu_{0}$. Moreover, by \eqref{eq:a}, \eqref{eq:b} we may write 
\[
t= 1-\frac{n+1}{1+\os}=\frac{\os-n}{1+\os}, \qquad \mu=\frac{n+1}{\om+1}-n=\frac{1-n\om}{\om+1}.
\]
Inserting for $t$ and $\mu=\mu_{0}$ in \eqref{eq:muiod}
leads to a quadratic equation for $\om$ in $\os$
with (the correct) solution $\rho_{2}$,
we omit the elementary calculation. 
Clearly, the argument can be read in reverse direction 
and the claims are indeed equivalent.
We remark that if $\mu=\mu_{0}$, as $\ov_{n+1}\to 1$ we compute 
$\uv_{n+1}\to 1$ and $\uv_{1}\to -2/(n-1)$, and one can 
show $\ov_{1}\to -1/n$.
We refine the metrical claim of Theorem~\ref{t1} in Section~\ref{ela}. 

Keep $g_{n}$ defined in \eqref{eq:hdef}.
An equivalent formulation of Theorem~\ref{t2} is the following.

\begin{theorem}  \label{t3}
	Let $n\geq 2$ be an integer. 
	Then precisely for triples $(s, \nu, \gamma)$ satisfying
	\begin{equation}  \label{eq:beyo}
	s\in [-n,0], \qquad \nu\in [-\frac{s}{n}, g_{n}(s)], \qquad 
	\gamma= (1-n)\frac{s+n\nu }{n+1+2s+(n-1)\nu}
	\end{equation}
	there is $\ux\in\mathbb{R}^n$ inducing equality in \eqref{eq:4} and
	with associated values
	\[
	\uv_{1}= s, \qquad \ov_{1}=\gamma, \qquad \ov_{n+1}=\nu.
	\]
	For given $(s, \nu, \gamma)$ obeying the restrictions \eqref{eq:beyo},
	the corresponding set 
	\begin{equation}  \label{eq:sima}
	\Sigma_{s,\nu}^{(n)}= \{ \ux\in\mathbb{R}^n: \; \uv_{1}=s, \qquad \ov_{1}=\gamma, \qquad
	\ov_{n+1}=\nu \}
	\end{equation}
	has packing dimension at least $1/2$, and positive 
	Hausdorff dimension as soon as $s>-n$.
\end{theorem}

The equivalence is derived similarly as for Theorem~\ref{t1}, we omit details.
If $\nu=\nu_{0}:= g_{n}(s)$, as $\uv_{1}\to -n$ we calculate $\ov_{1}\to -(n-1)/2$ and $\ov_{n+1}\to 1$, in fact $\uv_{n+1}\to 1$ holds. 
See again Section~\ref{ela} for 
refined metrical claims.

\subsection{Refinements of the metrical claims}  \label{ela}

For simplicity we state our metrical claims in the language of
Section~\ref{se41} only. Corresponding results in terms of classical exponents
can be inferred from \eqref{eq:a}, \eqref{eq:b}, which in particular 
imply the equivalences
\[
t\to 0^{+} \Longleftrightarrow \os\to n^{+} , \;
t\to 1^{-} \Longleftrightarrow \os\to\infty , \;
\mu=-\frac{t}{n} \Longleftrightarrow \om=\rho_{1} , \;
\mu=\mu_{0} \Longleftrightarrow \om=\rho_{2},
\]
and
\[
s\to 0^{-} \Leftrightarrow \om\to \frac{1}{n}^{+} , \;
s\to -n^{+} \Leftrightarrow \om\to\infty , \;
\nu=-\frac{s}{n} \Leftrightarrow \os=\tau_{1} , \;
\nu=\nu_{0} \Leftrightarrow \os=\tau_{2}.
\]
We start with refining Theorem~\ref{t1}. Let
\[
A=A_{t,\mu}^{(n)}= \frac{1-t}{2t+(n-1)\mu}\cdot
\frac{1}{n+1} \cdot \left(3t+2(n-1)\mu-n+\frac{n(1+t+(n-1)\mu)^2}{1-t}\right)\geq 0,
\]
with $A>0$ as soon as $t<1$, and derive $B=B_{t,\mu}^{(n)}$ and $C=C_{t,\mu}^{(n)}$ as
\begin{equation} \label{eq:bb}
B=n- \frac{(2-A)(n+1)}{n+1+2t+(n-1)\mu}, \qquad C= n-2+A\frac{n+1}{n+1+(n-1)t+n(n-1)\mu}.
\end{equation}
Then $A\in [0,2]$ and $n\geq B\geq n-2+A, n\geq C\geq n-2+A$
for every $t,\mu$ as in Theorem~\ref{t1}. Recall the Hausdorff dimension of a set never exceeds its packing dimension~\cite{falconer}.

\begin{theorem}  \label{thm100}
	The
	dimensions of the sets $\Theta=\Theta_{t,\mu}^{(n)}$ in \eqref{eq:dim}
	are bounded from below by
\begin{equation}  \label{eq:hd1}
\dim_{H}(\Theta)\geq n-2+ A, \qquad \dim_{P}(\Theta)\geq 
\max\{ B,C\}.
\end{equation}
There is equality (at least) if $t>0$ and $\mu=\mu_{0}$. Moreover, $\max\{ B,C\}\geq n-2+\frac{1}{n}$.
\end{theorem}

We conjecture equalities in \eqref{eq:hd1} as soon as $t>0$ and
$\mu\in [\mu_{0},-t/n)$, but
only rigorously prove it for $\mu=\mu_{0}$ in the last paragraph of Section~\ref{non1}.
The restriction $t>0$ is necessary and probably 
equality does not extend to $\mu=-t/n$ etiher, as we explain below.
We discuss special cases.
If $\mu$ attains its maximum value $\mu=-t/n$, then $A_{t,-t/n}^{(n)}$ decreases from $1+(n+1)^{-1}$ to $1$
as $t$ rises in $(0,1)$, and we iner
\[
\lim_{t\to 0^{+}} \dim_{P} \Theta_{t,-t/n}^{(n)}\geq
\lim_{t\to 0^{+}} \dim_{H} \Theta_{t,-t/n}^{(n)}\geq n-1+\frac{1}{n+1},
\quad \lim_{t\to 1^{-}} \dim_{H} \Theta_{t,-t/n}^{(n)}\geq n-1.
\]
As $t\to 0^{+}$ then $B_{t,-t/n}, C_{t,-t/n}$ tend 
to $n-1+1/(n+1)$ as well, suggesting
equality in all estimates.
We can pass to $t=0, t=1$ by considering limiting graphs, so
in particular
\begin{equation}  \label{eq:nmin1}
\dim_{P}(\Theta_{t,-t/n}^{(n)})\geq \dim_{H}(\Theta_{t,-t/n}^{(n)})\geq n-1,
\qquad t\in[0,1],
\end{equation}
The lower limit for the Hausdorff dimension is sharp, see the last paragraph of this section.
If $\mu=\mu_{0}$, evaluating the limits of $A_{t,\mu_{0} }^{(n)}, B_{t,\mu_{0} }^{(n)}, C_{t,\mu_{0} }^{(n)}$
as $t\to 0^{+}$ and $t\to 1^{-}$ from Theorem~\ref{thm100} we get
\[
\lim_{t\to 0^{+}}\dim_{P}(\Theta_{t, \mu_{0} }^{(n)}) =
\lim_{t\to 0^{+}}  \dim_{H}(\Theta_{t, \mu_{0} }^{(n)})
= n-2+\frac{3}{n+1},
\]
and
\begin{equation}  \label{eq:staefan}
\lim_{t\to 1^{-}}  \dim_{H}(\Theta_{t, \mu_{0} }^{(n)})= n-2, \qquad
\lim_{t\to 1^{-}}\dim_{P}(\Theta_{t, \mu_{0} }^{(n)})= n-2+\frac{1}{n}.
\end{equation}
On the other hand, if $t=0$ equivalent to $\ov_{j}=\uv_{j}=0$ for all $1\leq j\leq n+1$, we have
\begin{equation}  \label{eq:trivia}
\dim_{P}(\Theta_{0,0}^{(n)})=\dim_{H}(\Theta_{0,0}^{(n)})=n,
\end{equation}
since then $\os=n$, $\om=1/n$,
which holds for almost all $\ux\in \mathbb{R}^n$. More generally, we 
expect the maps $(t,\mu)\longmapsto \dim_{H}(\Theta_{t,\mu}^{(n)})$
and $(t,\mu)\longmapsto \dim_{P}(\Theta_{t,\mu}^{(n)})$
to be discontinuous on the curve $\mu=-t/n$ where $\sigma=0$. 
In this case our constructions in the proof 
below can be refined to obtain larger dimensions
than in Theorem~\ref{thm100}.

We turn towards Theorem~\ref{t3}.
Let 
\begin{equation}  \label{eq:D}
D=D_{s,\nu}^{(n)}:= n-\frac{s-s^2}{2s+(n-1)\nu}\geq 0
\end{equation}
and $D>0$ as soon as $s>-n$, and derive $E=E_{s,\nu}^{(n)}, F=F_{s,\nu}^{(n)}$ via
\begin{equation}  \label{eq:ee}
E= n - (n-D) \frac{(n+1)(s+(n-1)\nu+1)}{(1-s)(n+1+s-\nu)},
\qquad F= D\frac{n+1}{n+1+2s+(n-1)\nu}.
\end{equation}
Then $D\in [0,n]$ %
and $n\geq E\geq D, n\geq F\geq D$ for all $s,\nu$ in 
the parameter range. 

\begin{theorem}  \label{thm200}
The dimensions of the sets $\Sigma=\Sigma_{s,\nu}^{(n)}$
in \eqref{eq:sima} are bounded by
\begin{equation}  \label{eq:hd2}
\dim_{H}(\Sigma_{s,\nu}^{(n)}) \geq D, \qquad 
\dim_{P}(\Sigma_{s,\nu}^{(n)}) \geq \max\{ E,F\}.
\end{equation}
There is equality (at least) if $s<0$ and $\nu=\nu_{0}$. Moreover, $\max\{ E,F\}\geq \frac{1}{2}$.
\end{theorem}

We again conjecture equalities if $s<0$ and $\nu\in (-s/n,\nu_{0}]$, 
but can guarantee 
it only for $\nu=\nu_{0}$, see Section~\ref{non2}. Again $s<0$ 
and $\nu \neq -s/n$ are vital.
If $\nu=\nu_{0}$ then as $s\to 0$ all limits $D,E,F$
become $n-2+3/(n+1)$ so we get
\[
\lim_{s\to 0^{-}} \dim_{H}(\Sigma_{s,\nu_{0} }^{(n)}) =
\lim_{s\to 0^{-}} \dim_{P}(\Sigma_{s,\nu_{0} }^{(n)}) =
n-2+\frac{3}{n+1}.
\]
As $s$ increases in $[-n,0]$, the bound $D_{s,\nu_{0}}^{(n)}$ for 
the Hausdorff dimensions decays 
whereas $\max\{ E_{s,\nu_{0}}^{(n)},F_{s,\nu_{0}}^{(n)}\}$ increases. 
In the other extremal case $\nu=-s/n$, the values $D_{s,-s/n}^{(n)}$
are strictly increasing with $s$ and evaluating limits we get
\[
\lim_{s\to 0^{-}} \dim_{H}(\Sigma_{s,-s/n}^{(n)}) \geq
\lim_{s\to 0^{-}} D_{s,-s/n}^{(n)}=\lim_{s\to 0^{-}} E_{s,-s/n}^{(n)}= \lim_{s\to 0^{-}} F_{s,-s/n}^{(n)}=n-1+\frac{1}{n+1}, 
\]
with right hand side in $(0,1)$.
As $s\to -n$, we explain below that we must have
\begin{align}  \label{eq:riglim}
\lim_{s\to -n^{+}} \dim_{H}(\Sigma_{s,\nu}^{(n)})=\lim_{s\to -n^{+}} D_{s,\nu}^{(n)}=0, \qquad\qquad \nu\in [-s/n,\nu_{0}],
\end{align}
but we verify strictly positive dimensions as soon as $s<-n$. Furthermore
\[
\lim_{s\to -n^{+}}\dim_{P}(\Sigma_{s,\nu_{0} }^{(n)})= 
\lim_{s\to -n^{+}} F_{s,\nu_{0} }^{(n)}=
\frac{1}{2},
\]
and since $D,E,F$ decrease in $s, \nu$, we infer
the lower bound $1/2$ in Theorem~\ref{thm200}.
If $\nu=-s/n$ the limits as $s\to -n$ of both $E,F$ become $n-1+1/(n+1)$, so 
\[
\lim_{s\to -n^{+}}\dim_{P}(\Sigma_{s,-s/n }^{(n)})\geq 
n-1+\frac{1}{n+1}.
\]
Moreover, as soon as $s>-n$, the bounds for the dimensions become positive,
for every $\nu\in [-s/n,\nu_{0}]$. 
Again we have
\[
\dim_{H}(\Sigma_{0,0}^{(n)})=\dim_{P}(\Sigma_{0,0}^{(n)})=n
\]
and we expect the maps $(s,\nu)\longmapsto \dim_{H}(\Sigma_{s,\nu}^{(n)})$
and $(s,\nu)\longmapsto \dim_{P}(\Sigma_{s,\nu}^{(n)})$
to be discontinuous on the curve $\nu=-s/n$ where $\gamma=0$.

We justify \eqref{eq:riglim}.
If $s\to -n$ then $\om\to\infty$ by \eqref{eq:a},
and the sets $W_{w}^{(n)}:=\{\ux\in\mathbb{R}^n: \om(\ux)\geq w \}$
have Hausdorff dimension $(n+1)/(w+1)$ (see Jarn\'ik~\cite{jarnik})
which tends to $0$ as $w\to\infty$. On the other hand, $t\to 1$ gives $\os\to\infty$ and the according sets
$V_{w^{\ast}}^{(n)}:=\{\ux\in\mathbb{R}^n: \os(\ux)\geq w^{\ast} \}$ have Hausdorff dimension at least $n-1$ with equality for $w^{\ast}=\infty$, 
see~\cite{besicovich}.  
This shows the bound in \eqref{eq:nmin1} is optimal 
and explains the dimension drop from Theorem~\ref{thm100}
to Theorem~\ref{thm200}.
Our results complement metrical findings in~\cite{dfsu1},~\cite{dfsu2}.
For example the set of ''dually infinite singular vectors'' $S_{\infty^{\ast}}^{(n)}:=\{\ux\in\mathbb{R}^n: \wos(\ux)=\infty \}$ satisfies $\dim_{H}(S_{\infty^{\ast}}^{(n)})=n-2, \dim_{P}(S_{\infty^{\ast}}^{(n)})=n-1$, see~\cite[Section~1.2]{dfsu1}. Since
$t\to 1$ and $\mu=\mu_{0}$ imply $\wos\to \infty$ by a comment below Theorem~\ref{t0}, 
in this sense the limiting Hausdorff dimension in \eqref{eq:staefan} 
is as large as possible, whereas the packing limit is not.

\section{Proof of Theorem~\ref{t0}}  \label{se5}

We have seen that Theorem~\ref{t0} is equivalent to Theorem~\ref{t1}.
We split the proof of the latter
in existence and non-existence part. 

\subsection{Generalized existence result}  \label{mger}
For the existence part, we
show the following more general claim in the course of
Sections~\ref{mger}-\ref{dimen} that also includes Theorem~\ref{thm100}.

\begin{theorem} \label{thm1}
	Let $n\geq 2$ and $t\in[0,1]$. Derive
	\begin{equation}  \label{eq:tau}
	\mu_{0}=\mu_{0}(n,t)= \frac{(3-2n)t+1-2n+\sqrt{(1-t)[(4n-5)t+4n^2-4n+1]} }{2(n-1)^2}.
	\end{equation}
	Then $\mu_{0}\leq -t/n$ and for every $\mu\in [\mu_{0},-t/n]$ 
	there exists a non-empty set $\Theta^{\ast}=\Theta_{t,\mu}^{\ast (n)}\subseteq \mathbb{R}^n$ consisting of $\ux=\ux_{t,\mu}\in \mathbb{R}^{n}$ and
	\begin{equation} \label{eq:6}
	\ov_{n+1}=t , \quad \uv_{n+1}=\ov_{n}=  (1-n)\frac{t+n\mu }{n+1+2t+(n-1)\mu} , 
	\quad \uv_{1}=\cdots=\uv_{n-1}= \mu,
	\end{equation}
	and
	\begin{equation} \label{eq:JAP}
	\ov_{1}=\cdots=\ov_{n-1}= \frac{n\mu+t}{n+1+t-\mu},
	\end{equation}
	and
	\begin{equation} \label{eq:JEP}
	\uv_{n}= -\frac{1}{n} + \frac{n+1}{n}\cdot \frac{1-t}{1+n+(n-1)t+n(n-1)\mu}
	\end{equation}
	hold. In particular, \eqref{eq:6} implies
	equality in \eqref{eq:3} for any $\ux\in \Theta^{\ast}$.
	The dimensions of
	$\Theta_{t,\mu}^{\ast (n)}$
	satisfy the lower bounds \eqref{eq:hd1}, with equality if
	$t>0$ and $\mu=\mu_{0}$.
\end{theorem}

The identity $\ov_{1}=(n\mu+t)/(n+1+t-\mu)$ in \eqref{eq:JAP}
agrees with \eqref{eq:istg}
when using \eqref{eq:a}, so it is in fact necessary in our setting.
If $\mu=\mu_{0}$ then the values
in \eqref{eq:JAP}, \eqref{eq:JEP} coincide, thus $\ov_{1}=\cdots=\ov_{n-1}=\uv_{n}$. Moreover,
$\ov_{n+1}=t$ and $\uv_{1}=\mu_{0}$ and equality in 
\eqref{eq:3} directly imply 
all claims \eqref{eq:6}, \eqref{eq:JAP}, \eqref{eq:JEP}, in particular
by \eqref{eq:a} we may express $\wo$ as remarked in Section~\ref{se2}. 
Due to the additional conditions \eqref{eq:JAP}, \eqref{eq:JEP} we see
$\Theta_{t,\mu}^{\ast (n)}\subseteq \Theta_{t,\mu}^{(n)}$ with $\Theta_{t,\mu}^{(n)}$ from Theorem~\ref{t1}, so the latter
and Theorem~\ref{thm100} are indeed implied.
We will make use of the following calculations.

\begin{proposition} \label{pq}
	Let $n\geq 2$ be an integer. For any $t\in [0,1]$ and $\mu_{0}=\mu_{0}(n,t)$
	defined in \eqref{eq:tau} we have
	\[
	-\frac{t}{n}\geq \mu_{0}\geq -\frac{t^2 + (2n+1)t}{n^2-t} \geq -\frac{2}{n-1}t.
	\]
\end{proposition}

\begin{proof}
	We check the most challenging middle inequality first. Define	
	\[
	F_{n}(x,y)=(1+x+(n-1)y)^2 - (x-1)(y-1) = 0.
	\]
	We use that $\mu_{0}$ is solution to the quadratic equation
	\begin{equation} \label{eq:reflex}
	F_{n}(t,\mu_{0})=(1+t+(n-1)\mu_{0})^2 - (t-1)(\mu_{0}-1) = 0,
	\end{equation}
	equivalent to \eqref{eq:muiod}.
	Taking $A_{n}(t)=-(t^2+(2n+1)t)/(n^2-t)$ we see that
	for the resulting identity
	\[
	G_{n}(t) = F_{n}(t,A_{n}(t))= 0
	\]
	then leads to a quartic polynomial with solutions $t=0, t=1$
	and a double solution $t=-n\notin [0,1]$.
	Hence by the continuity of $\mu_{0}(n,t), A_{n}(t)$ and $F_{n}(t)$ in $t$
	we see that $F_{n}(t,A_{n}(t))$ and
	$\mu_{0}(n,t)-A_{n}(t)$ do not change sign on $[0,1]$ (as otherwise $\mu_{0}(n,t)=A_{n}(t)$
	for some $t\in(0,1)$ but then $F_{n}(t,\mu)=F_{n}(t,A_{n}(t))=G_{n}(t)=0$,
	contradiction to $G_{n}$ not having zeros in $(0,1)$). Thus 
	either $\mu_{0}\geq A_{n}(t)$ or $\mu_{0}\leq A_{n}(t)$ for all $t$ in $[0,1]$.
	
	To check that the first case occurs, it is convenient to use a different approach.
	The claim middle inequality, or $\mu_{0}\geq A_{n}(t)$, are equivalent to
	\[
	H_{n}(t)= \frac{(3-2n)t+1-2n+\sqrt{(1-t)[(4n-5)t+4n^2-4n+1]} }{2(n-1)^2}
	+ \frac{t^2 + (2n+1)t}{n^2-t} \geq 0.
	\]
	Then we calculate $H_{n}(0)=H_{n}(1)=0$ and
	\begin{align*}
	H_{n}^{\prime}(t)&= \frac{2n+1+2t}{n^2-t}+\frac{(t^2+(2n-1)t)}{(n^2-t)^2} +\frac{3}{2(n-1)^2}\\&- \frac{n}{(n-1)^2} - \frac{2n^2-4n+3+(4n-5)t}{2(n-1)^2 \sqrt{(1-t)[(4n-5)t+(2n-1)^2]} }
	\end{align*}
	and it is easily seen that
	$H_{n}^{\prime}$ has a pole at $t=1$ and the limit is $-\infty$
	as $t$ approaches $1$ from below. Thus $H_{n}(t)$ and hence 
	$\mu(n,t)-A_{n}(t)$ are non-negative 
	for $t\in [1-\epsilon,1]$, and by the above findings 
	actually for all $t\in [0,1]$.
	
	For the most left inequality, we can proceed very similarly 
	with $B_{n}(t)=-t/n$ instead of $A_{n}(t)$. Then $I_{n}(t)=F_{n}(t,B_{n}(t))=0$
	has solutions $-n, 0$ for $t$ and a similar argument shows $\mu_{0}(n,t)-B_{n}(t)$
	does not change sign on $[0,1]$, and by a similar derivative 
	argument we can again check the expression is never
	positive.
	Finally the most right inequality can be verified straight 
	forward using $t\in [0,1]$.
\end{proof}

By Theorem~\ref{dop}, 
it suffices to construct for given $t\in [0,1], \mu\in [\mu_{0},-t/n]$ an $n$-template $\textbf{P}=\textbf{P}_{t,\mu}$
inducing the upper and lower limit 
values $\uv_{j}, \ov_{j}$ as in Theorem~\ref{thm1}.
We may assume $0<t<1$ in the sequel. 
This follows from the compactness of the spectrum 
of $(\uv_{1},\ldots,\ov_{n+1})\subseteq \mathbb{R}^{2n+2}$ settled in~\cite{roytop} and the continuous dependency of the bounds
for $\sigma, \mu$ from $t$ in the theorem. 
The case 
$t=0$ implies $\uv_{j}=\ov_{j}=0$ for all $j$, which yields the 
trivial $n$-template
$P_{j}(q)=0$ for $1\leq j\leq n+1$, $q\in[0,\infty)$ anyway. It satisfies Theorem~\ref{thm1} and puts the origin in the
spectrum of $(\uv_{1},\ldots,\ov_{n+1})\in\mathbb{R}^{2n+2}$.
We should note that
Roy's spectrum differs from ours, 
but it is easily seen that compactness is preserved when switching
between the formalisms. 
In the sequel, we call $q$ a switch point
of $\textbf{P}$
if some $P_{j}$ is not differentiable at $q$, i.e. $P_{j}$ has a local
maximum or minimum.

\subsection{Preperiod of $\textbf{P}_{t,\mu}$}

As customary when applying the variational principle, we want to 
define an $n$-template with a periodic pattern. 
First we describe how to obtain
the initial state of the repeating construction. For given $t\in (0,1)$ 
and $\mu\in [\mu_{0},-t/n]$ with $\mu_{0}$ as in \eqref{eq:tau},
the goal is the following scenario: At some $q_{0}>0$ we
have
\begin{equation}  \label{eq:ida}
\frac{P_{1}(q_{0})}{q_{0}} =\cdots= \frac{P_{n-1}(q_{0})}{q_{0}}=\mu, \qquad \frac{P_{n}(q_{0})}{q_{0}} = \theta, \qquad \frac{P_{n+1}(q_{0})}{q_{0}} =t,
\end{equation}
which is the starting point of Figure~1 below.
Here according to \eqref{eq:sumv} we put
\begin{equation}  \label{eq:psit}
\theta= -(t+(n-1)\mu).
\end{equation}
Moreover, at $q_{0}$ we want the function $P_{n+1}$ to decay with slope
$-n$ and $P_{1},\ldots,P_{n}$ rise with slope $+1$.
It follows easily from the prescribed range for $\mu$ and 
Proposition~\ref{pq} that
\begin{equation} \label{eq:dieor}
-n\leq -\frac{2}{n-1}\leq \mu_{0}\leq \mu\leq -\frac{t}{n} \leq \theta\leq t\leq 1, \qquad t\in [0,1],
\end{equation}
so indeed the values $P_{j}(q_{0})/q_{0}$ belong to the required
interval $[-n,1]$, and their ordering
is as in \eqref{eq:ida}. 

We describe how the initial data \eqref{eq:ida} can be achieved.
Take any $q_{0}>0$. 
We start at $q=0$ with $P_{1}(0)=\ldots=P_{n}(0)=0$.
Let $P_{n+1}$ rise with slope $+1$ until some switch point $q^{\prime}\in(0,q_{0}]$, and then decay with slope $-n$ until $q_{0}$, where $q^{\prime}$ is chosen so that $P_{n+1}(q_{0})=tq_{0}$. 
By equating 
\[
P_{n+1}(q_{0}) = tq_{0} = q^{\prime} - n(q_{0}-q^{\prime}),
\]
we see
\[
q^{\prime}= \frac{n+t}{ n+1 } q_{0}.
\] 
To get the desired value
for $P_{n}(q_{0})$, we let $P_{n}$ together with $P_{1},\ldots,P_{n-1}$
initially
decay with slope $-1/n$ up to some switch point $q^{\prime\prime}$ where 
$P_{n}$ starts rising with slope $1$ up to $q_{0}$, with $q^{\prime\prime}$
chosen so that $P_{n}(q_{0}) = \theta q_{0}$. One determines
\[
q^{\prime\prime}= \frac{(1-\theta)n}{n+1} q_{0} \leq  q^{\prime},
\]
where the inequality holds due to \eqref{eq:dieor}.
We remark that $q^{\prime\prime}=0$ if $t=1$.
We take the remaining first $n-1$ successive minima functions $P_{1},\ldots,P_{n-1}$
identical in $[0,q_{0}]$, so that they are
determined in $[0,q_{0}]$ by the vanishing sum property \eqref{eq:sumv}
and the description of $P_{n}, P_{n+1}$ above.
This means $P_{1},\ldots,P_{n-1}$ decay with slope $-1/n$ in $[0,q^{\prime\prime}]$,
with slope $-2/(n-1)$ in $[q^{\prime\prime}, q^{\prime}]$
and finally rise with slope $+1$ in $[q^{\prime}, q_{0}]$.
This concludes the preperiod.

\subsection{Period of $\textbf{P}_{t,\mu_{0}}$: Special case $\mu=\mu_{0}$} \label{exi}

We take $\mu=\mu_{0}$ throughout this section and consequently 
impliclty consider $\theta=\theta_{0}$ derived from \eqref{eq:psit} 
with this choice. 
We remark that by our choice of $\mu_{0}$, we verify that
$t,\mu=\mu_{0},\theta=\theta_{0}$ in \eqref{eq:psit} 
are linked by the quadratic identity
\begin{equation}  \label{eq:rr}
\theta^2 - 2\theta -\mu t+t+\mu= (\theta-1)^2 - (\mu-1)(t-1)=0,
\end{equation}
which reflects \eqref{eq:muiod}.

The figure shows the first period $[q_{0},q_{1}]$ 
of the iterative construction, 
starting from $q_{0}$ where \eqref{eq:ida} holds, up to some
$q_{1}>q_{0}$ to be determined below.
We continue the slopes as from the preperiod at $q_{0}$ 
to the right, i.e.
slope $-n$ for $P_{n+1}$ and $+1$ for $P_{1},\ldots,P_{n}$. 
At some point $\tilde{q}_{1}>q_{0}$ the functions $P_{n}, P_{n+1}$
will intersect.
By equating
\[
P_{n+1}(\tilde{q}_{1})= tq_{0}-n(\tilde{q}_{1}-q_{0}) =\theta q_{0}+(\tilde{q}_{1}-q_{0})=P_{n}(\tilde{q}_{1}),
\]
and a brief computation, upon inserting \eqref{eq:psit}, 
this point is given by
\begin{equation}  \label{eq:tilq1}
\tilde{q}_{1}= \frac{n+1+2t+(n-1)\mu}{n+1}\cdot q_{0}
\end{equation}
and induces the quotients
\begin{equation}  \label{eq:observed}
\frac{P_{n}(\tilde{q}_{1})}{\tilde{q}_{1}}=\frac{P_{n+1}(\tilde{q}_{1})}
{\tilde{q}_{1}}= 
(1-n)\frac{t+n\mu }{n+1+2t+(n-1)\mu}.
\end{equation}
Note that the right hand side of \eqref{eq:observed}
is the desired slope for $\uv_{n+1}$ and $\ov_{n}$ in \eqref{eq:6}.
At $\tilde{q}_{1}$ the functions $P_{n}, P_{n+1}$ exchange slopes, so $P_{n}$
starts decaying with slope $-n$ and $P_{n+1}$ rising with slope $+1$.
Then at some point $\tilde{q}_{2}>\tilde{q}_{1}$
the rising functions $P_{1}=\cdots=P_{n-1}$ will intersect $P_{n}$. 
From
\begin{equation}  \label{eq:qqt}
P_{n}(\tilde{q}_{2})= 
tq_{0} - n(\tilde{q}_{2}-q_{0}) = \mu q_{0} + (\tilde{q}_{2}-q_{0})=P_{1}(\tilde{q}_{2})\cdots = P_{n-1}(\tilde{q}_{2})
\end{equation}
we derive 
\begin{equation}  \label{eq:rqt}
\tilde{q}_{2} = \frac{t-\mu+n+1}{n+1}q_{0}.
\end{equation}
To the right of this
switch point $\tilde{q}_{2}$, let the slopes of $P_{1},\ldots, P_{n-1}$ become $-2/(n-1)$
whereas the slope of $P_{n}$ becomes $+1$. Observe $P_{n+1}$ also still
increases with slope $+1$. Since we can assume $t<1$, then
at some point $q_{1}\geq \tilde{q}_{1}$ 
we will have $P_{n+1}(q_{1})/q_{1}=t$, concretely
from imposing $P_{n+1}(q_{1})=\theta q_{0}+(q_{1}-q_{0})=q_{1}t$
we calculate
\begin{equation} \label{eq:q1}
q_{1}= \frac{\theta-1}{t-1} q_{0} =\frac{ \mu-1 }{ \theta-1 } q_{0},
\end{equation}
where the right equality uses $\mu=\mu_{0}, \theta=\theta_{0}$ 
and reflects \eqref{eq:rr}.
With  a little effort it can be checked, and follows from the below calculations,
that actually $q_{1}\geq \tilde{q}_{2}$. Indeed by some
rearrangements this is equivalent to
\[
t^2+(2n+1-\mu)t+n^2 \mu \geq 0 \qquad 
\Longleftrightarrow \qquad \mu\geq -\frac{t^2 + (2n+1)t}{n^2-t},
\]
which is confirmed in Proposition~\ref{pq}.
Moreover, at the same point $q_{1}$ we further have
\[
\frac{P_{n}(q_{1})}{q_{1}} = 
\frac{\mu q_{0}+(q_{1}-q_{0})}{q_{1}}=
\theta, \qquad
\frac{P_{1}(q_{1})}{q_{1}} = \cdots = \frac{P_{n-1}(q_{1})}{q_{1}} = \mu.
\]
The left identity is equivalent to \eqref{eq:q1}, the right 
follows consequently
from $P_{n+1}(q_{1})/q_{1}=t$, the
vanishing sum property \eqref{eq:sumv} and $P_{1}(q_{1})=\cdots=P_{n-1}(q_{1})$.
%
Finally we let $q_{1}$ be a switch point where $P_{n+1}$ starts
decaying with slope $-n$, whereas $P_{1},\ldots,P_{n-1}$ start
rising with slope $+1$. Note $P_{n}$ is differentiable at $q_{1}$ 
with slope $+1$ and clearly $q_{1}>q_{0}$ if $t>0$, which we can assume.

Since the slopes as well as the right-sided derivatives of all $P_{j}$
at $q_{1}$ are then identical to the data at $q_{0}$,
at $q_{1}$ we have precisely the same conditions as at $q_{0}$. Thus,
up to a scaling factor $q_{1}/q_{0}$ in each step,
we can extend the construction from $[q_{0},q_{1}]$ 
periodically ad infinitum. 
Together with the preperiod, this will give an
$n$-template on $[0,\infty)$.

\begin{tikzpicture}

\draw[thick,->] (0,0) -- (13,0) node[anchor=west] {q};
\draw[thick,->] (0,-4) -- (0,6) node[anchor=south] {P(q)};

\draw[dotted,->] (0,0) -- (11,6) node[anchor=west] 
{$\overline{\varphi}_{n+1}=t $};

\draw[dotted,->] (0,0) -- (11,3) node[anchor=west] {$\theta_{0}$};

\draw[dotted,->] (0,0) -- (10.8,-3.27) node[anchor=west] {$\underline{\varphi}_{1}=\cdots=\underline{\varphi}_{n-1}=\mu_{0} $};

\draw[dotted,->] (0,0) -- (11,3.5) node[anchor=west] {$ \underline{\varphi}_{n+1}=\overline{\varphi}_{n}=\sigma $};

\draw[dotted,->] (0,0) -- (10.8,-0.38) node[anchor=west] {$\overline{\varphi}_{1}=\cdots=\overline{\varphi}_{n-1}=\underline{\varphi}_{n} $};

\draw (4,1.09) -- (8,4.35) node[anchor=south east] {$P_{n+1}$};

\draw [line width=0.45mm] (5.2,-0.2) -- (4,-1.225)  node[anchor=north east] {$P_{1}=\cdots=P_{n-1}$};

\draw (5.2,-0.2) -- (8,2.2) node[anchor=north west] {$P_{n}$} ;

\draw (5.2,-0.2) -- (4,2.2) node[anchor=south east] {$P_{n+1}$};

\draw  [line width=0.45mm] (5.2,-0.2) -- (8,-2.4) node[anchor=north east] {$P_{1}=\cdots=P_{n-1}$};

\fill[black] (4,0) circle (0.06cm) node[anchor=south] {$q_{0}$};

\fill[black] (4.4,0) circle (0.06cm) node[anchor=south] {$\tilde{q}_{1}$};

\fill[black] (5.2,0) circle (0.06cm) node[anchor=south] {$\tilde{q}_{2}$};

\fill[black] (8,0) circle (0.06cm) node[anchor=south] {$q_{1}$};

\node at (3.85,0.75) { $P_{n}$};

\node at (7,0.95) { $+1$};

\node at (7,3.15) { $+1$};

\node at (7.1,-1.1) { $-\frac{2}{n-1}$};

\node at (4.7,-1) { $+1$};

\node at (5,1) { $-n$};

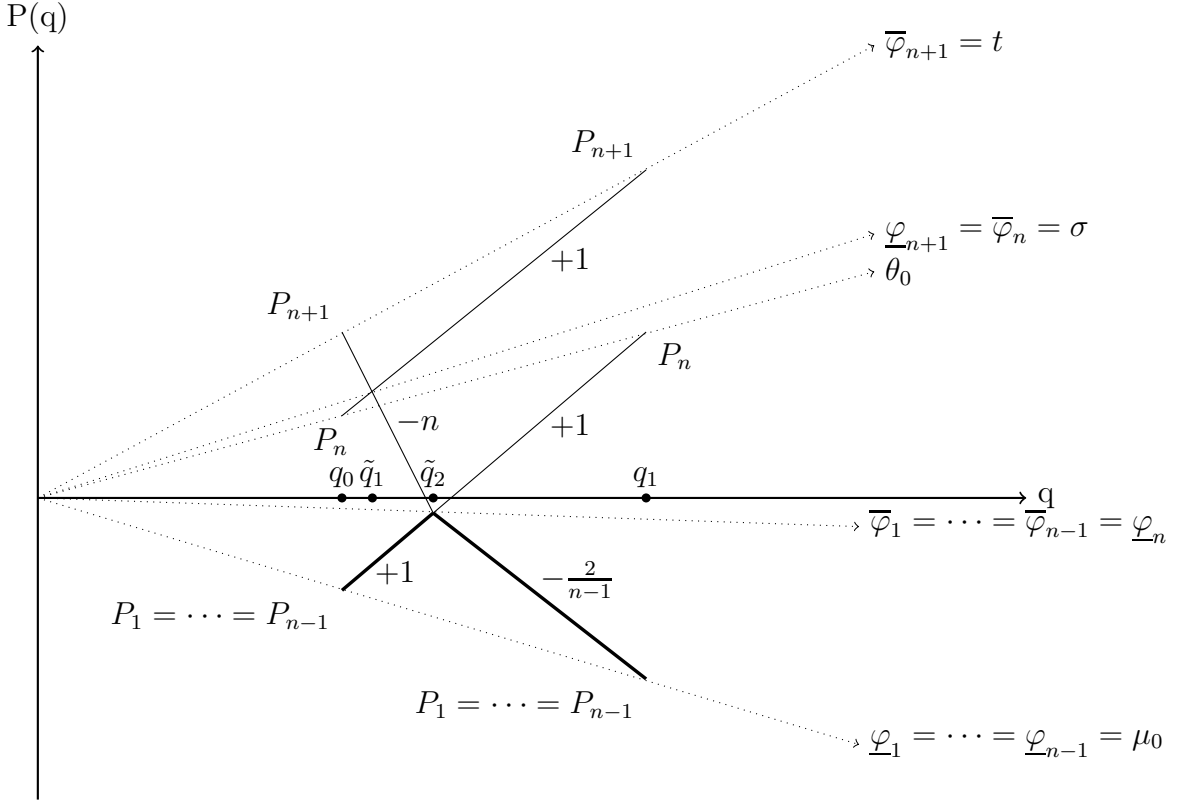
\node at (5.5,-4.5) {Figure 1: Sketch period of $\textbf{P}_{t,\mu_{0}}$ (special case $\mu=\mu_{0}$)};

\end{tikzpicture}


\subsection{Period of $\textbf{P}_{t,\mu}$: The general case}  \label{gc}

Now let $\mu$ be arbitrary in $[\mu_{0},-t/n]$.
We start the period construction as in the special case $\mu=\mu_{0}$:
After the preperiod we derive at some $q_{0}>0$ 
where for $\theta=-(t+(n-1)\mu)$ we are given
\[
\frac{P_{1}(q_{0})}{q_{0}} =\cdots= \frac{P_{n-1}(q_{0})}{q_{0}}=\mu, \qquad \frac{P_{n}(q_{0})}{q_{0}} = \theta, \qquad \frac{P_{n+1}(q_{0})}{q_{0}} =t.
\] 

We describe the period $[q_{0},q_{1}]$. Starting from $q_{0}$,
we let $P_{n+1}$ decay with slope $-n$ and the others rise
with slope $+1$ until at $\tilde{q}_{1}$ the functions 
$P_{n}, P_{n+1}$ meet and exchange slopes. Then at some point
$\tilde{q}_{2}>\tilde{q}_{1}$ the functions $P_{1}=\cdots=P_{n-1}$
meet $P_{n}$. The calculations for $\tilde{q}_{1},\tilde{q}_{2}$
and $P_{n}(\tilde{q}_{2})$ are precisely
as in the special case in Section~\ref{exi}, 
for general $\mu$ and implied $\theta$ throughout.

Now at $\tilde{q}_{2}$ we make a twist to our construction above
by prescribing that all $P_{1},\ldots,P_{n}$ decay with slope $-1/n$
in $[\tilde{q}_{2}, \tilde{q}_{3}]$ for some $\tilde{q}_{3}\geq \tilde{q}_{2}$ to be determined, while $P_{n+1}$ keeps increasing with slope $+1$. At some point $q_{1}>q_{0}$ we will have $P_{n+1}(q_{1})/q_{1}= t$
again. It is obvious that $q_{1}>\tilde{q}_{1}$  
and since $\mu\geq \mu_{0}$ it follows from the 
argument in the special case and $\mu\geq \mu_{0}$
that even $q_{1}\geq \tilde{q}_{2}$. The identity \eqref{eq:observed}
and the left identity in \eqref{eq:q1} 
hold for the same reason as in the special case, in particular 
$q_{1}$ is evaluated as for $\mu=\mu_{0}$.

We will choose $\tilde{q}_{3}\leq q_{1}$ and
in the interval $[\tilde{q}_{3}, q_{1}]$ we let $P_{n}$ rise with slope
$+1$ and $P_{1}, \ldots,P_{n-1}$ decay with slope $-2/(n-1)$.
(In particular, if $\tilde{q}_{2}=\tilde{q}_{3}$ then the graph
is as for $\mu=\mu_{0}$). We claim that
upon a proper choice of $\tilde{q}_{3}\in [\tilde{q}_{2}, q_{1}]$, at $q_{1}$ again we will have
\[
\frac{P_{1}(q_{1})}{q_{1}} =\cdots= \frac{P_{n-1}(q_{1})}{q_{1}}=\mu, \qquad \frac{P_{n}(q_{1})}{q_{1}} = \theta, \qquad \frac{P_{n+1}(q_{1})}{q_{1}} =t.
\]
Assume this is shown.
Then the period is finished and again we repeat the construction
upon scaling by $q_{1}/q_{0}$ in each step.

To evaluate $\tilde{q}_{3}$, we have to satisfy the identity
\[
\theta q_{1}=
P_{n}(\tilde{q}_{2}) - \frac{1}{n} (\tilde{q}_{3}- \tilde{q}_{2}) + (q_{1}- \tilde{q}_{3}),
\]
thus inserting for $\theta$ we get
\[
\tilde{q}_{3}= \frac{n}{n+1}\left[ P_{n}(\tilde{q}_{2})+\frac{1}{n}\tilde{q}_{2}+(1+t+(n-1)\mu)q_{1} \right].
\]
Inserting from \eqref{eq:qqt}, \eqref{eq:rqt}, \eqref{eq:q1} we get
\begin{align} \label{eq:tilq3}  
\tilde{q}_{3} &= \frac{n}{n+1}\left[ t+n+\frac{1-n}{n} (t-\mu+n+1)+\frac{(1+t+(n-1)\mu)^2}{1-t} \right]\cdot q_{0} \nonumber \\
&= \frac{1+t+(n-1)\mu}{n+1}\cdot \frac{1+n+(n-1)(t+n\mu)}{1-t} \cdot q_{0}.
\end{align}
We now finally check analytically 
that indeed $\tilde{q}_{3}\in [\tilde{q}_{2}, q_{1}]$ as claimed for
any choice $n\geq 2$, $t\in(0,1)$ and $\mu\in[\mu_{0},-t/n]$.

Writing $q_{1}/q_{0}=(1-\theta)/(1-t)$ and $\tilde{q}_{3}/q_{0}=(1-\theta)(1+n(1-\theta)/(1-t))/(n+1)$, the claim
$\tilde{q}_{3}\leq q_{1}$ can be rearranged to $\theta\geq -t/n$ which is true by \eqref{eq:dieor}. For the estimate $\tilde{q}_{2}\leq \tilde{q}_{3}$,
we check by inserting that there is equality 
$\tilde{q}_{2}= \tilde{q}_{3}$ at the minimum value $\mu=\mu_{0}$,
which agrees with the special case from the section above. 
Since $\tilde{q}_{2}$ decreases with slope $-1/(n+1)$ 
as a function of $\mu$ by \eqref{eq:rqt}, to conclude
it suffices to check that for fixed $t$ the value $\tilde{q}_{3}$ 
in \eqref{eq:tilq3} increases
as a function of $\mu$. We calculate
\[
\frac{d}{d\mu} \tilde{q}_{3} = \frac{(n-1)[(2n-1)t+2n+1+2n(n-1)\mu]}{(n+1)(1-t)}q_{0}.
\]
Since $t\in [0,1]$, for this to be positive we require
\[
(2n-1)t+2n+1+2n(n-1)\mu \geq 0.
\]
Using $\mu\geq \mu_{0}\geq -\frac{2}{n-1}t$  by Proposition~\ref{pq}
and $t\in[0,1]$ this can be readily verified.
This completes the construction of first period $[q_{0},q_{1}]$ 
for general $\mu$.
The periodic continuation to an $n$-template on $[0,\infty)$
via adjacent scaled copies is performed as in Section~\ref{exi}.

\begin{tikzpicture}

\draw[thick,->] (0,0) -- (13,0) node[anchor=west] {q};
\draw[thick,->] (0,-4) -- (0,6) node[anchor=south] {P(q)};

\draw[dotted,->] (0,0) -- (11,6) node[anchor=west] {$\overline{\varphi}_{n+1} =t$};

\draw[dotted,->] (0,0) -- (11,1) node[anchor=west] {$\theta$};

\draw[dotted,->] (0,0) -- (11,2.05) node[anchor=west] {$\underline{\varphi}_{n+1}=\overline{\varphi}_{n}=\sigma$};

\draw[dotted,->] (0,0) -- (11,-2.5) node[anchor=west] {$\underline{\varphi}_{1}=\cdots=\underline{\varphi}_{n-1}=\mu $};

\draw[dotted,->] (0,0) -- (11,-0.4) node[anchor=west] {$\overline{\varphi}_{1}=\cdots=\overline{\varphi}_{n-1}$};

\draw[dotted,->] (0,0) -- (11,-1.4) node[anchor=west] {$\underline{\varphi}_{n}$};

\draw (5.5,0.5) -- (9.5,5.2) node[anchor=south] {$P_{n+1}$};

\draw [line width=0.7mm] (6.4,-0.25) -- (8,-1)  node[anchor=west] {$P_{1}=\cdots=P_{n}$};

\draw (8,-1) -- (9.5,0.85) node[anchor=south] {$P_{n}$} ;

\draw (6.4,-0.25) -- (5.5,3) node[anchor=south] {$P_{n+1}$};

\draw [line width=0.4mm] (6.4,-0.25) --  (5.5,-1.25)  node[anchor=north east] {$P_{1}=\cdots=P_{n-1}$};

\draw [line width=0.4mm] (8,-1) -- (9.5,-2.15) node[anchor=north east] {$P_{1}=\cdots=P_{n-1}$};


\fill[black] (5.5,0) circle (0.06cm) node[anchor=south] {$q_{0}$};

\fill[black] (6.04,0) circle (0.06cm) node[anchor=south] {$\tilde{q}_{1}$};

\fill[black] (6.4,0) circle (0.06cm) node[anchor=south] {$\tilde{q}_{2}$};

\fill[black] (8,0) circle (0.06cm) node[anchor=south] {$\tilde{q}_{3}$};

\fill[black] (9.5,0) circle (0.06cm) node[anchor=south] {$q_{1}$};


\node at (5.2,0.7) { $P_{n}$};

\node at (7.2,3) { $+1$};

\node at (8.7,0.3) { $+1$};

\node at (7.1,-1) { $-\frac{1}{n}$};

\node at (9.3,-1.5) { $-\frac{2}{n-1}$};

\node at (5.4,-0.8) { $+1$};

\node at (5.4,1.7) { $-n$};

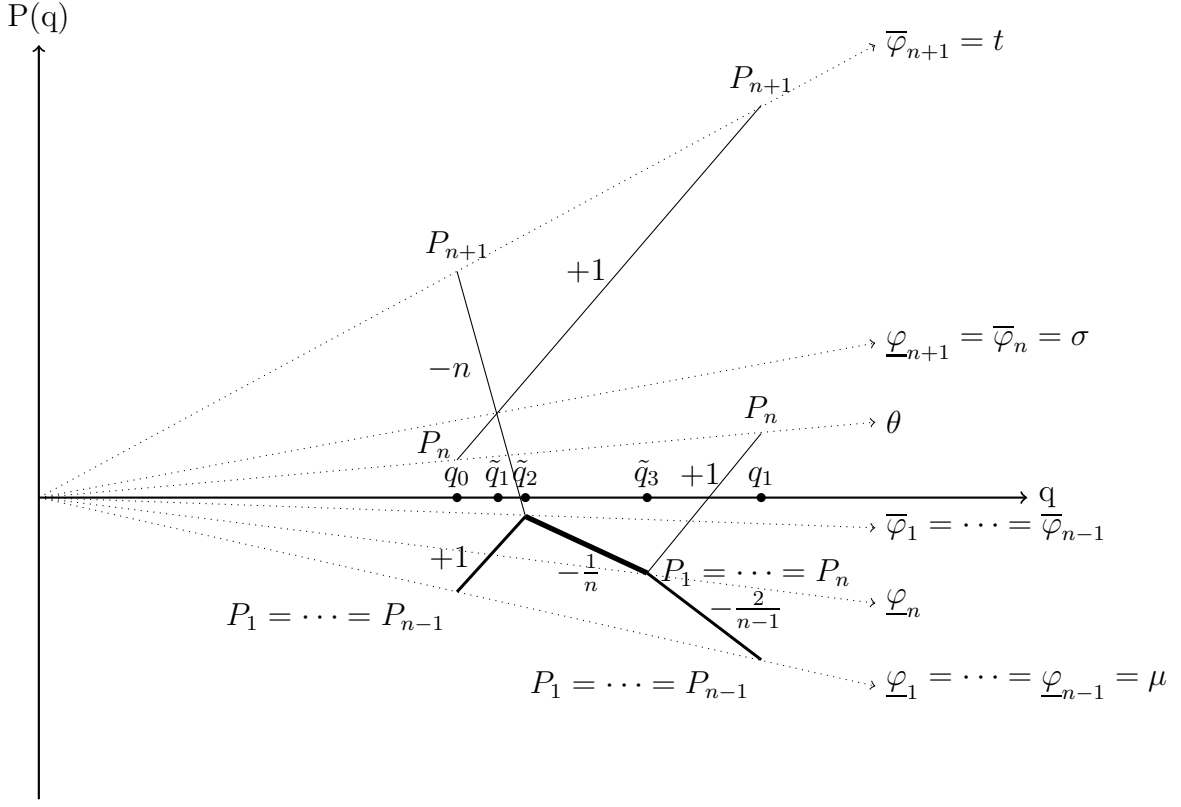
\node at (5.5,-4.5) {Figure 2: Sketch period of $\textbf{P}_{t,\mu}$ (general case)};

\end{tikzpicture}

\subsection{Evaluation of extremal values}

Since the maximum slope $P_{n+1}(q)/q$ of $P_{n+1}$
for $q\in [q_{0},q_{1}]$ is clearly attained
at the interval ends where it takes the value $t$, and similarly
the minimal slopes for $P_{1}=\cdots=P_{n-1}$ at the interval ends
equal to $\mu$, we see that 
\[
\liminf_{q\to\infty} \frac{P_{j}(q)}{q}= \mu, \; (1\leq j\leq n-1), \qquad 
\limsup_{q\to\infty} \frac{P_{n+1}(q)}{q}= t.
\] 
Moreover, obviously the expression $P_{n+1}(q)/q$ for $q\in [q_{0},q_{1}]$ is minimal at its local minimum $\tilde{q}_{1}$, and similarly the slope of $P_{n}$ attains its
maximum within $[q_{0},q_{1}]$ at $\tilde{q}_{1}$. 
From Theorem~\ref{dop} and \eqref{eq:idf}, \eqref{eq:landp}, \eqref{eq:observed} we conclude \eqref{eq:6}. 

Finally, for \eqref{eq:JAP}, \eqref{eq:JEP} we show that within $q\in[q_{0},q_{1}]$, the  values $P_{1}(q)/q=\cdots=P_{n-1}(q)/q$
take their maxima at $\tilde{q}_{2}$, and $P_{n}(q)/q$ its minimum at $\tilde{q}_{3}$. 
It is clear that the extrema in question are taken in the interval $I=[\tilde{q}_{2},\tilde{q}_{3}]$, since outside the slopes
take the extremal values $-n$ and $1$. We show that 
$P_{n}(\tilde{q}_{2})/\tilde{q}_{2}= (n\mu+t)/(n+1+t-\mu)$ 
exceeds the slope $-1/n$ of $P_{1},\ldots,P_{n}$
in $I$. Then the values $P_{j}(q)/q$ for $j=1,2,\ldots,n$ decrease within $I$ 
and the claim follows. To show 
\[
\frac{n\mu+t}{n+1+t-\mu} > -\frac{1}{n}
\]
we rearrange to the equivalent form $n+1>(n-1)t+ (n^2+1)\mu^2$ 
which is true since
as $t\in [0,1]$ and $\mu\leq 0$ by Proposition~\ref{pq} we have
\[
n+1 > n-1 \geq (n-1)t \geq (n-1)t+ (n^2+1)\mu^2.
\]
The claim \eqref{eq:JAP} follows directly.
For \eqref{eq:JEP}, from \eqref{eq:qqt}, \eqref{eq:rqt}, \eqref{eq:tilq3} we calculate
\[
\uv_{n}=\frac{P_{n}(\tilde{q}_{3})}{\tilde{q}_{3}} =
\frac{P_{n}(\tilde{q}_{2}) - \frac{\tilde{q}_{3}-\tilde{q}_{2}}{n}}{\tilde{q}_{3}}=
-\frac{1}{n} + \frac{n+1}{n}\cdot \frac{1+t+(n-1)\mu}{1+t+(n-1)\mu+\frac{n(1+t+(n-1)\mu)^2}{1-t}}.
 \]
Dividing numerator and denominator by $\theta=1+t+(n-1)\mu$
yields the claimed expression after a brief rearrangement.
We again conclude with Theorem~\ref{dop} and \eqref{eq:idf}, \eqref{eq:landp}.
Finally, inserting for $\underline{\varphi}_{1}, \underline{\varphi}_{n+1}, \ov_{n+1}$ from \eqref{eq:6},
a calculation verifies equality in \eqref{eq:3}.

\subsection{Deduction of metrical results}  \label{dimen}

We bound the Hausdorff and packing dimensions of the
set $\Theta_{t,\mu}^{\ast (n)}$ in Theorem~\ref{thm1}
as in \eqref{eq:hd1}. First assume $0<t<1$ again where our
construction is well-defined. 
Since any set $\Theta_{t,\mu}^{\ast (n)}$
is contained in $\Theta_{t,\mu}^{(n)}$ from Theorem~\ref{t1}, 
clearly \eqref{eq:hd1} follows.
We determine the contraction rates for 
the $n$-template $\mathbf{P}=\textbf{P}_{t,\mu}$
constructed above. 
We evaluate the local contraction rates within the period interval
$[q_{0},q_{1}]$ as
\begin{equation}  \label{eq:dieslopes}
\delta(\mathbf{P},q)=  \begin{cases}
n, & \text{if}\ q\in [q_{0}, \tilde{q}_{1}], \\
n-1, & \text{if}\ q\in [\tilde{q}_{1}, \tilde{q}_{3}], \\ 
n-2, & \text{if}\ q\in [\tilde{q}_{3}, q_{1}]. \\
\end{cases}
\end{equation}
Denoting for $j\geq 1$ the $j$-th period 
interval $I_{j}=[q_{j-1},q_{j}]$, this is true accordingly in $I_{j}$. 
From the variational principle we directly conclude that
the Hausdorff and packing dimensions cannot be less than $n-2$.
For the precise calculation, we observe that the local rate
decays within each interval $I_{j}$. 
We readily conclude that the lower 
limit is attained when considering intervals $[q_{0},q_{N}]$, 
and in fact by periodicity the resulting
average contraction rate in these intervals 
is independent of $N\geq 1$. So from the variational principle 
Theorem~\ref{deep} we get
\begin{align*}
\dim_{H}(\Theta_{t,\mu}^{\ast (n)}) \geq \underline{\delta}(\mathbf{P}) &= \frac{\int_{q_{0}}^{q_{1}} \delta(\mathbf{P},q)}{q_{1}-q_{0}}=
\frac{  n(\tilde{q}_{1}-q_{0}) + (n-1)(\tilde{q}_{3}-\tilde{q}_{1}) 
+(n-2)(q_{1}-\tilde{q}_{3})  }{q_{1}-q_{0}}   \\
&= \frac{(n-2)q_{1} + \tilde{q}_{1} + \tilde{q}_{3} - nq_{0} }{q_{1}-q_{0}}
= n-2+ \frac{\tilde{q}_{1} + \tilde{q}_{3}-2q_{0} }{q_{1}-q_{0}}.
\end{align*} 
Inserting for $\tilde{q}_{1},\tilde{q}_{3},q_{1}$ from
\eqref{eq:tilq1}, \eqref{eq:q1}, \eqref{eq:tilq3} we verify
$\dim_{H}(\Theta_{t,\mu}^{\ast (n)}) \geq n-2+A$ 
as in \eqref{eq:hd1}.
For the upper limit, we consider intervals $[q_{0},\tilde{q}_{1,N}]$
and $[q_{0},\tilde{q}_{3,N}]$
for large $N$, where $\tilde{q}_{i,N}$ denotes for $i=1,2,3$ the 
value corresponding to $\tilde{q}_{i}\in [q_{0},q_{1}]$ in the
interval $[q_{N-1},q_{N}]$. In particular $q_{N}\leq \tilde{q}_{1,N}\leq \tilde{q}_{3,N}\leq q_{N+1}$ for all $N$ and the contraction rates in the subintervals
$[q_{N}, \tilde{q}_{1,N}]$, $[\tilde{q}_{1,N}, \tilde{q}_{3,N}]$
and $[\tilde{q}_{3,N},q_{N+1}]$
take the values as for $N=0$ in \eqref{eq:dieslopes}. Hence
\[
\dim_{P}(\Theta_{t,\mu}) \geq  
\overline{\delta}(\mathbf{P}) \geq \max\{ S,T\} 
\]
where $S$ and $T$ are respecitvely the average limit contraction rates
in the intervals $[q_{0}, \tilde{q}_{N,1}]$ and $[q_{0}, \tilde{q}_{N,3}]$
respectively as $N\to\infty$. In fact it is not hard to check
equality $\overline{\delta}(\mathbf{P})= \max\{ S,T\}$.
To conclude,
we show $S\geq B, T\geq C$ with
$B,C$ as in \eqref{eq:bb}. Since we identified $\underline{\delta}(\mathbf{P})$
as the average contraction rate in any interval $[q_{0},q_{N}]$ and $q_{0}=o(q_{N})$ as $N\to\infty$, we evaluate
\begin{align*}
S&= \lim_{N\to\infty} \frac{ \int_{q_{0}}^{\tilde{q}_{1,N}} \delta(\mathbf{P},q)\; dq }{ \tilde{q}_{1,N}-q_{0}  } =\lim_{N\to\infty} \frac{ \underline{\delta}(\mathbf{P})(q_{N}-q_{0}) + n(\tilde{q}_{1,N}-q_{N}) }{\tilde{q}_{1,N}-q_{0}} \\ &=
 \lim_{N\to\infty} \frac{ \underline{\delta}(\mathbf{P})q_{N} + n(\tilde{q}_{1,N}-q_{N}) }{\tilde{q}_{1,N}} =
n - (n-\underline{\delta}(\mathbf{P}))\cdot \lim_{N\to\infty} \frac{q_{N}}{ \tilde{q}_{1,N} }.
\end{align*}
Now since $q_{N}/\tilde{q}_{1,N}$
is independent of $N$, inserting $\underline{\delta}(\mathbf{P})\geq n-2+A$ we infer
\[
S\geq n- \frac{q_{0}}{ \tilde{q}_{1} } (n-\underline{\delta}(\mathbf{P}))
\geq n - (2-A)\frac{q_{0}}{ \tilde{q}_{1} } = n- \frac{(2-A)(n+1)}{n+1+2t+(n-1)\mu}= B.
\]
For $T$ a similar caclulation shows
\begin{align*}
T&= \lim_{N\to\infty} \frac{ \int_{q_{0}}^{\tilde{q}_{3,N}} \delta(\mathbf{P},q)\; dq }{ \tilde{q}_{3,N}-q_{0}  } 
\\ &= \lim_{N\to\infty} \frac{ \int_{q_{0}}^{ q_{N+1}} \delta(\mathbf{P},q)\; dq - \int_{ \tilde{q}_{3,N} }^{ q_{N+1}} \delta(\mathbf{P},q)\; dq  }{ \tilde{q}_{3,N}-q_{0}  }\\ &=
\lim_{N\to\infty} \frac{(q_{N+1}-q_{0}) \underline{\delta}(\mathbf{P}) -(n-2)(q_{N+1}-\tilde{q}_{3,N})}{\tilde{q}_{3,N}-q_{0}}\\
&= 
\lim_{N\to\infty} \frac{q_{N+1} \underline{\delta}(\mathbf{P}) -(n-2)(q_{N+1}-\tilde{q}_{3,N})}{\tilde{q}_{3,N}}\\
&= n-2 + (\underline{\delta}(\mathbf{P})-n+2)\lim_{N\to\infty} \frac{q_{N+1}}{\tilde{q}_{3,N}}
\\ &= n-2+ (\underline{\delta}(\mathbf{P})-n+2)\frac{q_{1}}{\tilde{q}_{3}}=
n-2+A\frac{q_{1}}{\tilde{q}_{3}}\\
&= n-2+A\frac{(n+1)(1+t+(n-1)\mu)}{(1-t)(1+t+(n-1)\mu)+n(1+t+(n-1)\mu)^2}\\
&=n-2+A\frac{n+1}{n+1+(n-1)t+n(n-1)\mu}=C,
\end{align*}
as claimed, where we inserted for $q_{1}, \tilde{q}_{3}$ 
from \eqref{eq:q1}, \eqref{eq:tilq3} in the last line.
Finally, for $t=0$ the claim \eqref{eq:hd1}
is trivial by \eqref{eq:trivia}, and
we can extend the formula to $t=1$ by considering a limiting $n$-template,
compare with~\cite[Section~2]{ichcomm},
we omit details. 
The proof of Theorem~\ref{thm1} is complete.

\subsection{Extending the range of $\wos$} \label{ee}

We sketch how to alter the graphs in Figure~2 to get a prescribed value
for $\wos$ as in the interval of 
\eqref{eq:A1}. We keep $q_{0}, \tilde{q}_{1}$ and the graph from 
$\textbf{P}_{t,\mu}$
in $[q_{0}, \tilde{q}_{1}]$ unchanged. We alter the formulas
for $\tilde{q}_{2}, \tilde{q}_{3}, q_{1}$, still satisfying $\tilde{q}_{1}\leq \tilde{q}_{2}\leq \tilde{q}_{3}\leq q_{1}$,
and introduce
a new point $\tilde{r}$ between $\tilde{q}_{1}$ and $\tilde{q}_{2}$.
We let $P_{n}, P_{n+1}$ decay with slope $-(n-1)/2$ in $[\tilde{q}_{1},\tilde{r}]$ and then starting at $\tilde{r}$
we let $P_{n+1}$ rise with slope $+1$ and $P_{n}$ decay with slope $-n$
until it meets $P_{1}=\dots=P_{n-1}$. The construction in $[\tilde{q}_{2},q_{1}]$ remains basically as in $\textbf{P}_{t,\mu}$ 
in Figure~2.
For given
\[
\eta\in \left[0,(1-n)\frac{t+n\mu }{n+1+2t+(n-1)\mu}\right],
\]
appropriate choices of $\tilde{r},\tilde{q}_{2},\tilde{q}_{3}, q_{1}$ induce
an $n$-template $\textbf{P}_{t,\mu,\eta}$ satisfying \eqref{eq:6} apart from $\uv_{n+1}$ altered to $\uv_{n+1}= \eta$.
Thus by \eqref{eq:b} we obtain any $\wos$ as in \eqref{eq:A1} (remark: \eqref{eq:JAP}, \eqref{eq:JEP} are not preserved). 
We omit the calculations and only want to illustrate qualitatively the
graph in Figure~3 below. We omit metrical claims derived 
from Theorem~\ref{deep} as well. 

\begin{tikzpicture}

\draw[thick,->] (0,0) -- (13,0) node[anchor=west] {q};
\draw[thick,->] (0,-4) -- (0,6) node[anchor=south] {P(q)};

\draw[dotted,->] (0,0) -- (11,6) node[anchor=west] {$\overline{\varphi}_{n+1} =t$};

\draw[dotted,->] (0,0) -- (11,1.1) node[anchor=west] {$\theta$};

\draw[dotted,->] (0,0) -- (11,2.03) node[anchor=west] {$\overline{\varphi}_{n}$};

\draw[dotted,->] (0,0) -- (11.5,1.44) node[anchor=west] {$\underline{\varphi}_{n+1}= \eta$};

\draw[dotted,->] (0,0) -- (11,-3.55) node[anchor=west] {$\underline{\varphi}_{1}=\cdots=\underline{\varphi}_{n-1}=\mu $};

\draw[dotted,->] (0,0) -- (11,-0.4) node[anchor=west] {$\overline{\varphi}_{1}=\cdots=\overline{\varphi}_{n-1}$};

\draw[dotted,->] (0,0) -- (11,-1.5) node[anchor=west] {$\underline{\varphi}_{n}$};

 \draw (5.5,0.55) -- (6.04,1.12) ;



\draw (6.04,1.12) -- (5.5,3) node[anchor=south] {$P_{n+1}$};


\draw [line width=0.45mm] (6.65,-0.25) --  (5.5,-1.75)  node[anchor=north east] {$P_{1}=\cdots=P_{n-1}$};

\draw [line width=0.45mm] (7.85,-1.1) -- (9.5,-3.06) node[anchor=north east] {$P_{1}=\cdots=P_{n-1}$};

\draw (7.85,-1.1) -- (9.5,0.96) node[anchor=south] {$P_{n}$} ;

\draw [line width=0.7mm] (6.65,-0.25) -- (7.85,-1.1) ;

\draw  [line width=0.35mm] (6.04,1.12) -- (6.3,0.8) ;

\draw (6.65,-0.25) -- (6.3,0.8) ;

\draw (9.5,5.2) -- (6.3,0.8) ;

\fill[black] (5.5,0) circle (0.06cm) node[anchor=south] {$q_{0}$};

\fill[black] (6.04,0) circle (0.06cm) node[anchor=south] {$\tilde{q}_{1}$};

\fill[black] (6.65,0) circle (0.06cm) node[anchor=south] {$\tilde{q}_{2}$};

\fill[black] (6.3,0) circle (0.06cm) node[anchor=south] {$\tilde{r}$};

\fill[black] (7.85,0) circle (0.06cm) node[anchor=south] {$\tilde{q}_{3}$};

\fill[black] (9.5,0) circle (0.06cm) node[anchor=south] {$q_{1}$};

\node at (5.4,2) {$-n$};

\node at (8.15,2.8) {$+1$};

\node at (6.2,-1.4) {$+1$};

\node at (9.2,-2) {$-\frac{2}{n-1}$};

\node at (7.1,-1.1) {$-\frac{1}{n}$};

\node at (8.65,-0.6) {$+1$};

\node at (6.37,1.45) {\tiny$-\frac{n-1}{2}$};

\node at (5.35,0.8) {$+1$};

\node at (6.75,0.65) {$-n$};

\node at (9.5,5.6) {$P_{n+1}$};

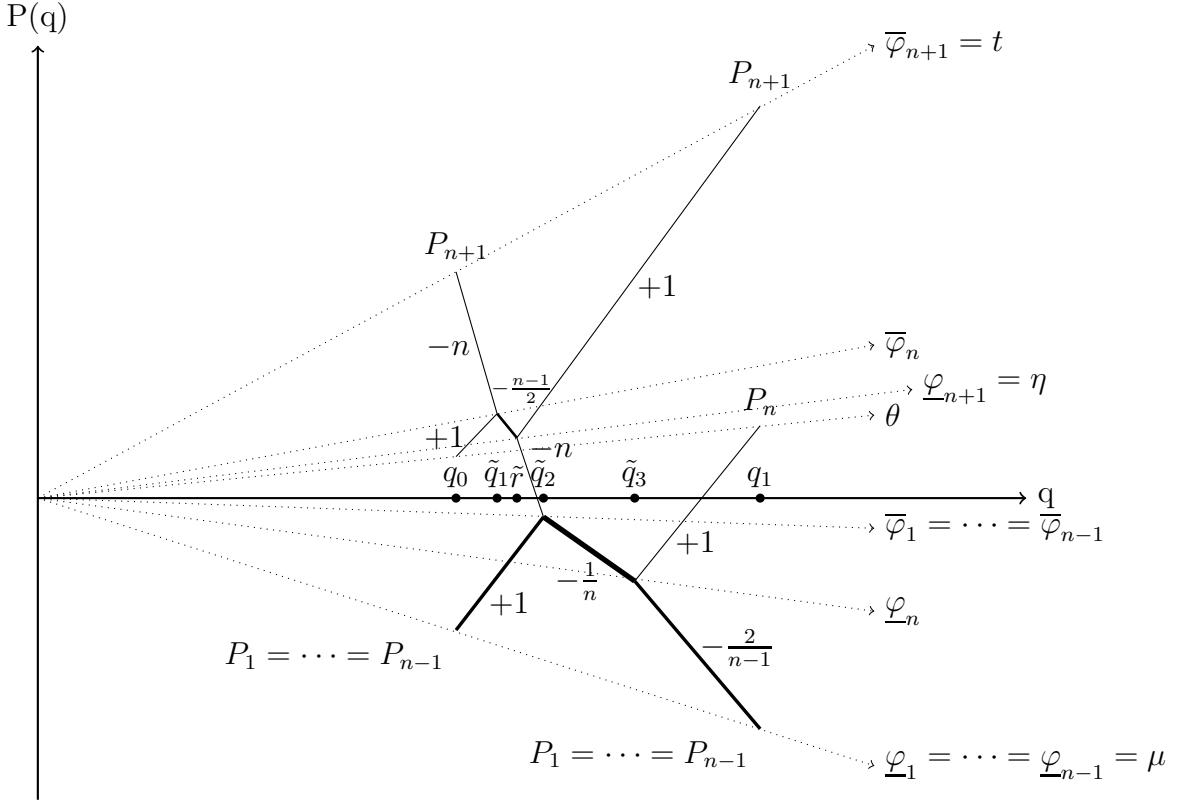
\node at (5.5,-4.5) {Figure 3: Sketch period $\textbf{P}_{t,\mu,\eta}$ (extended case)};

\end{tikzpicture}

\subsection{Non-existence in Theorem~\ref{t1}}  \label{non1}

To establish the non-existence part of Theorem~\ref{t1}
means to show the following claim. 

\begin{theorem}  \label{t10}
	Let $n\geq 2$, $t\in [0,1]$ and $\mu_{0}=\mu_{0}(n,t)$ as in 
	Theorem~\ref{t1}. Then for $\mu\notin [\mu_{0},-t/n]$, the
	set $\Theta_{t,\mu}^{(n)}$ of $\ux\in\mathbb{R}^n$ that induces $\ov_{n+1}=t, \uv_{1}=\mu$
	and equality in \eqref{eq:3} is empty. In other words, no $\ux\in\mathbb{R}^n$
	induces
	\begin{equation} \label{eq:01}
	\ov_{n+1}=t, \qquad \uv_{n+1}=\sigma:= (1-n)\frac{t+n\mu }{n+1+2t+(n-1)\mu}, \qquad \uv_{1}=\mu.
	\end{equation}
\end{theorem}

We prove the theorem. 
For $\mu>-t/n$ we cannot even have $\ov_{n+1}=t, \uv_{1}=\mu$ due to the
reverse inequality $\uv_{1}\leq -\ov_{n+1}/n$ in \eqref{eq:khin2}.
It remains to contradict $\mu<\mu_{0}$ upon the assumptions
$\ov_{n+1}=t, \uv_{1}=\mu$ and equality in \eqref{eq:3} of the theorem. 
Keep in mind for the sequel the equivalence in the claims of
Theorem~\ref{t10}, i.e. upon
$\ov_{n+1}=t, \uv_{1}=\mu$, equality in \eqref{eq:3}
is equivalent to $\uv_{n+1}$ taking the value $\sigma$ in \eqref{eq:01}.

\underline{Step 1}: We show that
equality in \eqref{eq:3} implies that essentially the situation as in the
interval $[q_{0},\tilde{q}_{1}]\subseteq [q_{0},q_{1}]$ in 
Figure~2 (or Figure 1)
occurs for arbitrarily large $q_{0}$.
For this we basically rephrase an argument within the proof  of~\cite[Theorem~3.2]{ichnyj}: 
Since $\ov_{n+1}=t$, for any $\epsilon>0$ there are arbitrarily
large $q_{0}$ with $|L_{n+1}(q_{0})-tq_{0}| \leq \epsilon q_{0}$. 
Choose large $q_{0}$ with this property.
For simplicity of notation, 
we omit $\epsilon$ and use $o$ notation in the sequel,
so we write $L_{n+1}(q_{0})=tq_{0}+o(q_{0})$ and mean that
in fact we consider a sequence of $q_{0}$ values with this property that
tends to infinity. 
We may assume that
at $q_{0}$ there is a local maximum of $L_{n+1}$. Consider the next
point $q_{0}+\tilde{q}$ where $L_{n}, L_{n+1}$ meet to the right of $q_{0}$,
i.e. $\tilde{q}>0$ minimal so that $L_{n}(q_{0}+\tilde{q})= L_{n+1}(q_{0}+\tilde{q})$. By definition of $\uv_{n+1}$ clearly 
\[
L_{n+1}(q_{0}+\tilde{q})\geq (\uv_{n+1}-\epsilon)(q_{0}+\tilde{q})=(\sigma-o(1)) (q_{0}+\tilde{q}).
\]
Since $L_{n}$ has slope at most $1$, we infer
\begin{equation} \label{eq:kukuk}
L_{n}(q_{0}) \geq L_{n}(q_{0}+\tilde{q})- \tilde{q}
= L_{n+1}(q_{0}+\tilde{q})- \tilde{q}
\geq (\sigma-1)\tilde{q} +\sigma q_{0} - o(q_{0}+\tilde{q}).
\end{equation}
Together with the bounded sum property \eqref{eq:bsum}, we infer
\begin{equation}  \label{eq:weinf}
L_{1}(q_{0}) \leq - \frac{L_{n}(q_{0})+L_{n+1}(q_{0})}{n-1} + O(1) \leq
-\frac{(t+\sigma)q_{0} + (\sigma-1)\tilde{q} }{n-1} + o(q_{0}+\tilde{q}). 
\end{equation}
We estimate $\tilde{q}$. Since $L_{n+1}$ decays with slope $-n$ 
in $[q_{0},q_{0}+\tilde{q}]$ and $L_{n+1}(q_{0}+\tilde{q})/(q_{0}+\tilde{q})$
is at least $\sigma+o(1)$ by definition of $\uv_{n+1}=\sigma$, we have
\[
L_{n+1}(q_{0}+\tilde{q}) = L_{n+1}(q_{0}) -n \tilde{q} \geq 
(\sigma-o(1)) (q_{0}+\tilde{q}).  
\]
Inserting $L_{n+1}(q_{0})= tq_{0}+o(1)q_{0}$ we get
\begin{equation}  \label{eq:specht}
\tilde{q} \leq \left(\frac{ t-\sigma }{n+\sigma}+o(1)\right) q_{0}.
\end{equation}
Since $\sigma\leq 1$, by \eqref{eq:weinf} when dividing by $q_{0}$ we get 
 \[
 \frac{L_{1}(q_{0})}{q_{0}}  \leq
 -\frac{t+\sigma + (\sigma-1)\frac{ t-\sigma }{n+\sigma} }{n-1} + o(1).
 \]
 As we can assume $\uv_{1}\leq L_{1}(q_{0})/q_{0}+o(1)$,
 after some rearrangement when taking limits we may drop the $o(1)$ terms, 
 and find the corresponding inequality
 \[
 \uv_{1}\leq - \frac{ \ov_{n+1}+\uv_{n+1}+(\uv_{n+1}-1)
 	\frac{ \ov_{n+1}-\uv_{n+1} }{ n+\uv_{n+1} } }{n-1}
 \]
 to be equivalent to \eqref{eq:3}.
 This means that in case of equality
 in \eqref{eq:3}, there must be (asymptotic) equality in all inequalities
 above. So as $q_{0}$ as above tends to infinity, by \eqref{eq:kukuk}, \eqref{eq:specht} we must have
 \[
 L_{n+1}(q_{0}) = tq_{0}+o(q_{0}), \qquad L_{n}(q_{0})= \left((\sigma-1)\frac{ t-\sigma }{n+\sigma}+ \sigma\right) q_{0}+ o(q_{0})
 \]
 and further from equality in \eqref{eq:weinf} we infer
 \[
 L_{j}(q_{0})= -\frac{t+\sigma + (\sigma-1)\frac{ t-\sigma }{n+\sigma} }{n-1} q_{0}+ o(q_{0}), \qquad 1\leq j\leq n-1.
 \]
 With some calculation, we check that when dropping the remainder terms,
 the expression for 
 $L_{n}(q_{0})/q_{0}$ agrees with the value $\theta$ from
\eqref{eq:psit}, and $L_{1}(q_{0})/q_{0}$ with $\mu$. Upon
identifying $\tilde{q}+q=\tilde{q}_{1}$,
 this indeed verifies that essentially the combined graph
 in the interval $[q_{0},\tilde{q}_{1}]$ must look 
 like in Figure~2 
 from the construction.
 
\underline{Step 2}: we show that if $\mu<\mu_{0}$ we cannot extend the graph of Figure~2 from $[q_{0}, \tilde{q}_{1}]$ to the right of $\tilde{q}_{1}$ without violating the requirements of a combined graph, 
 thereby we get a contradiction.
 Let $r>\tilde{q}_{1}$ be the first coordinate of
 the next meeting point of $L_{n}, L_{n+1}$
 to the right of $\tilde{q}_{1}$, i.e. the smallest solution for
 $L_{n}(r)=L_{n+1}(r)$ with $r>\tilde{q}_{1}$. Write $I=[\tilde{q}_{1},r]$.
 Now we distinguish two cases.\\
 \underline{Case 1}: The functions $L_{n-1}$ and $L_{n}$ do not meet in $I$, i.e. $L_{n}(q)>L_{n-1}(q)$ for $q\in I$.
 Then it is clear from the theory of combined graphs/$n$-templates
 that, up to $o(q)$, the graph in $I$ must look as follows: 
 there is some switch point $u\in I$ so that 
 in $[\tilde{q}_{1},u]$ the function $L_{n+1}$ must rise with slope $1$
 and decay in $[u,r]$ with slope $-n$, whereas for some $v=u+o(u)$ 
 very close to $u$ the opposite happens for $L_{n}$, i.e. 
 $L_{n}$ decays with slope $-n$ in $[\tilde{q}_{1},v]$ and  
 increases with slope $+1$ in $[v,r]$. (The functions $L_{1},\ldots,L_{n-1}$ all rise
 with average slope $+1-o(1)$ in the entire interval $I$.)
 We justify this claim, but for brevity omit full
 rigorosity: First note that since $L_{n-1}$ and $L_{n}$ do not meet in the interior of $I$, any ''serious'' local minimum of $L_{n}(q)$ at some $q\in I$ induces a local maximum of $L_{n+1}(\ell)$ at some $\ell=q+O(1)$. This can be seen by passing to a close $n$-template 
 as in \eqref{eq:jaroy} and the convexity
 condition in Definition~\ref{defnt} and \eqref{eq:bsum}. Now since
 $L_{n}$ and $L_{n+1}$ move apart in a neighborhood to the right of $\tilde{q}_{1}$, the function $L_{n}$ must change slope to $+1$ somewhere
 in $I$, and by the above argument in proximity $L_{n+1}$
 must change slope to $-n$. So clearly there is at least one switch
 point in the interior of $I$ where $L_{n}, L_{n+1}$ exchange slopes as above.
 Assume there was another ''serious'' switch point $\tilde{\ell}$ in the interior of $I$ where $L_{n}$ changes slope. Then at $\tilde{\ell}$,
 $L_{n}$ starts to decay with slope $-n$
 and $L_{n+1}$ must start to rise with slope $+1$ by
 \eqref{eq:bsum} and we assume these slopes
 continue to the right on a subinterval of $I$ of
 substantial length. Since $\tilde{\ell}$ is in the interior of $I$,
 then in an associated close $n$-template satisfying \eqref{eq:jaroy},
 the function $P_{n+1}$ would have
 a local minimum which is not a local maximum of $P_{n}$. This contradicts
 the convexity condition of templates again. (Instead of passing to templates in the last step, we can 
 alternatively argue with the first two successive minima of the dual lattice point problem). This confirms our claim.
 
 In particular, the average slope of $L_{n}$ and $L_{n+1}$ in $I$ is $-(n-1)/2+o(1)<0$. Since clearly $L_{n+1}(q)\geq 0$ everywhere, 
 hence at $r$ we get $L_{n+1}(r)/r< L_{n+1}(\tilde{q}_{1})/\tilde{q}_{1}= \uv_{n+1}+o(1)$, contradiction to $\uv_{n+1}\leq L_{n+1}(r)/r-o(1)$
 unless $r-\tilde{q}_{1}$ is very small. However, it is clear that
 we may assume this is not the case. For example, we may pass
 to $n$-templates again, or start with $\varepsilon>0$ and 
 restrict to points $r$
 with $r>(1+\varepsilon)\tilde{q}_{1}$ and then use the above argument. 
 We omit the technical details.
 
 \underline{Case 2}:  The functions $L_{n-1}$ and $L_{n}$ meet in $I$. 
 Starting at its local minimum $\tilde{q}_{1}$ where it meets $L_{n}$, 
 the function $L_{n+1}$ rises with slope 
 $+1$. However, since $\ov_{n+1}=t$, this happens 
 at most until a point $y$ on the first axis 
 where $L_{n+1}(y)/y= t+o(1)$. 
 Then $L_{n+1}$ decays with slope $-n$ until it meets $L_{n}$ at
 $r>y>\tilde{q}_{1}$. Since $L_{n-1}$ and $L_{n}$ meet
 in $I$ and no slope can exceed $+1$, it is clear that 
 \begin{equation}  \label{eq:ti}
 L_{n}(r) \leq L_{n-1}(q_{0}) + (r-q_{0}).
 \end{equation}
 Recall that in the construction for $\mu=\mu_{0}$, the value
 $y$ was as large as possible (up to $o(1)$) 
 since $L_{n+1}$ indeed went up until $y$ with $L_{1}(y)/y=t$,
 and there was equality
 in \eqref{eq:ti}, since at the meeting point $\tilde{q}_{2}$ 
 of $L_{n-1}$ and $L_{n}$ the slope of $L_{n}$ changed from $-n$ to $+1$
 and remained $+1$ until it met $L_{n+1}$. 
 Further identifying our $y$ with $q_{1}$
 from that proof, for $\mu=\mu_{0}$ we had the 
 minimum possible value $L_{n+1}(r)/r= \sigma$ at $r$.
 So it is geometrically obvious that
if we start with $\mu<\mu_{0}$ (which also implies larger
values of $\sigma$ and $\theta$), even in the most disadvantageous 
 case of maximal $y$ and equality in \eqref{eq:ti}, 
 at the smallest point $z>y$ where $L_{n+1}(z)/z= \sigma$, we will have $L_{n+1}(z)\geq (1+\varepsilon) L_{n}(z)$ with 
 some $\varepsilon>0$ depending on $\mu, \mu_{0}$.
 We omit the explicit calculation. This means that $r>(1+\epsilon_{0})z$ 
 and $L_{n+1}$
 continues to decay with slope $-n$ in $[z,r]$ until it meets $L_{n}$ at $r$,
 for some $\epsilon_{0}>0$.
 Thus obviously $\varphi_{n+1}(r)=L_{n+1}(r)/r\leq (1-\epsilon_{1}) L_{n+1}(z)/z= (1-\epsilon_{1})\sigma$
 for some $\epsilon_{1}>0$. However, this contradicts the definition
 of $\sigma=\uv_{n+1}$.
 This completes the proof of Theorem~\ref{t10}.
 
 We finally observe that the equality in the dimension formulas
 \eqref{eq:hd1} for $t>0$ and $\mu=\mu_{0}$ follows from the proof above. 
 Our argument shows that
 then the combined graph must indeed be composed
 from consecutive periods as in Figure~1, up to
 $o(q)$ as $q\to\infty$. 
 Take the family $\mathscr{F}$ of $n$-templates
 with these properties, which is closed under
 finite perturbations in view of the error term. 
 Then it is not hard to see that the suprema of $\underline{\delta}(\textbf{Q}), \overline{\delta}(\textbf{Q})$ over $\textbf{Q}\in \mathscr{F}$ are attained for $\textbf{Q}=\textbf{P}_{t,\mu}$ as constructed (since $o(q)$ has
 a negligible effect in the limit 
 and by changing slopes of some $P_{j}$ locally in intervals where
 consecutive functions $P_{j}$ are glued, we may 
 only decrease the local contraction rate. We skip details).
 Application of the variational principle yields the claim.
 Note that we lose the case $t=0$, where
\eqref{eq:hd1} indeed fails as pointed out in Section~\ref{ela}, 
since then $q_{0}=q_{1}$ in our 
construction by \eqref{eq:q1}, so the period 
collapses to a singleton and we get no $n$-template.

\section{Proof of Theorem~\ref{t2}}  \label{se6}

By equivalence of Theorems~\ref{t2}, \ref{t3}
we again may just prove Theorem~\ref{t3}, and we show
the following more general existence claim that includes 
Theorem~\ref{thm200}.

\begin{theorem}  \label{thm2}
	Let $n\geq 2$ and $s\in [-n,0]$. Derive $\nu_{0}=g_{n}(s)$ 
	with $g_{n}$ as in \eqref{eq:hdef}.
	Then $\nu_{0}\geq -s/n$ and for every $\nu\in [-s/n,\nu_{0}]$
	there exists a non-empty set $\Sigma^{\ast}=\Sigma_{s,\nu}^{\ast (n)}$ consisting of
	$\ux=\ux_{s,\nu}\in \mathbb{R}^{n}$ whose associated
	quantities $\uv_{j},\ov_{j}$ satisfy
	\begin{equation}  \label{eq:tier}
	\uv_{1}=s , \qquad \ov_{1}=\uv_{2}= (1-n)\frac{s+n\nu }{n+1+2s+(n-1)\nu}, 
	\qquad \ov_{3}=\cdots=\ov_{n+1}= \nu,
	\end{equation}
	and
	\begin{equation} \label{eq:JAPh}
	\uv_{3}=\cdots=\uv_{n+1}= \frac{n\nu+s}{n+1+s-\nu},
	\end{equation}
	and
	\[
	\ov_{2}=  -\frac{1}{n} + \frac{n+1}{n}\cdot \frac{1-s}{1+n+(n-1)s+n(n-1)\nu}.
	\]
	Every $\ux\in \Sigma_{s,\nu}^{\ast (n)}$ induces equality in \eqref{eq:4}.
	The dimensions of $\Sigma_{s,\nu}^{\ast (n)}$ are bounded
	as in \eqref{eq:hd2}, with equality if $s<0$ and $\nu=\nu_{0}$.
\end{theorem}

The identity $\uv_{n+1}=(n\nu+s)/(n+1+s-\nu)$ in \eqref{eq:JAPh}
agrees with \eqref{eq:istg2}
when using \eqref{eq:b}, so again it is necessary in our framework.
We again have $\Sigma_{s,\nu}^{\ast (n)} \subseteq \Sigma_{s,\nu}^{(n)}$
with $\Sigma_{s,\nu}^{(n)}$ from Theorem~\ref{t3}.
Moreover, if $\nu=\nu_{0}$, then $\ov_{2}=\uv_{3}=\cdots=\uv_{n+1}$
and additional equality in \eqref{eq:4} induces all values $\uv_{j}, \ov_{j}$ 
as in the theorem.
Again we construct suitable $n$-templates $\textbf{P}_{s,\nu}$ 
in order to apply Theorem~\ref{dop}.
The construction is dual in some sense. We omit certain computations
that are similar to the proof of Theorem~\ref{thm1}. We start with the
dual version of Proposition~\ref{pq}.

\begin{proposition}  \label{qp}
	For any $s,\nu_{0}$ as in Theorem~\ref{thm2} we have
	\[
	 -\frac{2}{n-1} s\geq \nu_{0} \geq -\frac{s^2 + (2n+1)s}{n^2-s} \geq -\frac{s}{n}.
	\]
\end{proposition}

We skip the proof as it works very similar as in Proposition~\ref{pq}.
We explain how we construct $n$-templates $\textbf{P}_{s,\nu}$ with the desired properties.
We may again assume strict inequalities 
$-n<s<0$ by compactness of the spectrum.

\subsection{Preperiod of $\textbf{P}_{s,\nu}$}

Similarly to Theorem~\ref{thm1}, here for some $q_{0}>0$ we want
\begin{equation}  \label{eq:hir}
\frac{P_{3}(q_{0})}{q_{0}} =\cdots= \frac{P_{n+1}(q_{0})}{q_{0}}=\nu, \qquad \frac{P_{2}(q_{0})}{q_{0}} = \vartheta, \qquad \frac{P_{1}(q_{0})}{q_{0}} =s,
\end{equation}
where again $\vartheta$ is determined from $s, \nu$ 
in view of \eqref{eq:sumv} via
\begin{equation}  \label{eq:tv}
\vartheta= -(s+(n-1)\nu).
\end{equation}
Moreover, at $q_{0}$ the functions $P_{1}, P_{2}$ rise with slope $+1$
while $P_{3},\ldots,P_{n+1}$ decay with slope $-2/(n-1)$.
By Proposition~\ref{qp}
we again check $-n\leq s\le \vartheta\leq -s/n\leq  \nu\leq \nu_{0}\leq 1$ for any $s\in[-n,0]$. 
So the slopes are well-defined and the ordering \eqref{eq:hir}
is correct. Moreover, $s=-n$ is equivalent
to $\vartheta=\nu=1$. 

By choice of $\vartheta$, for $\nu=\nu_{0}=g_{n}(s)$ we again have 
\[
\vartheta^2 - 2\vartheta-\nu s+\nu+s= (\vartheta-1)^2 - (\nu-1)(s-1)=0.
\]
To obtain \eqref{eq:hir}, starting at $q=0$ 
we let $P_{1}$ decay with slope $-n$ up to a switch
point $q^{\prime}\in(0,q_{0}]$ where it starts increasing with slope $+1$ until $q_{0}$.
Hereby $q^{\prime}$ is determined via the property $P_{1}(q_{0})= sq_{0}$,
giving $q^{\prime}=((1-s)/(n+1))q_{0}$.
In $[0,q^{\prime}]$ we let all $P_{2},\ldots,P_{n+1}$ rise with slope $+1$.
At the switch point $q^{\prime}$ we start letting $P_{2}$ decay with slope 
$-n$ up to some point $q^{\prime\prime}\geq q^{\prime}$ while the other functions
all rise with slope $+1$ in $[q^{\prime}, q^{\prime\prime}]$.
Then starting from $q^{\prime\prime}$ we let $P_{2}$ rise with slope $+1$ 
so that $P_{3},\ldots,P_{n+1}$ have slopes $-2/(n-1)$ in $[q^{\prime\prime},q_{0}]$. A suitable choice of $q^{\prime\prime}$
will lead to $P_{2}(q_{0})/q_{0} = \vartheta$, 
and by $P_{1}(q_{0})/ q_{0}=s$, the 
vanishing sum property \eqref{eq:sumv} and $P_{3}(q_{0})=\cdots=P_{n+1}(q_{0})$, actually all conditions in \eqref{eq:hir} are implied. 
Concretely $q^{\prime\prime}=((2-s-\vartheta)/(n+1))q_{0}$ 
is derived
from
\[
q^{\prime} -  n(q^{\prime\prime}-q^{\prime}) + (q_{0}-q^{\prime\prime})=
\vartheta q_{0}
\]
and inserting for $q^{\prime}$, and
indeed
$q^{\prime}\leq q^{\prime\prime}$ since this is equivalent to
$\vartheta\leq 1$ which is trivial.

\subsection{Period of $\textbf{P}_{s,\nu}$ and conclusion}  \label{cons}

It is convenient to give a reverse construction of the period,
i.e. start from $q_{1}>q_{0}$ where the properties
\begin{equation}  \label{eq:hir0}
\frac{P_{3}(q_{1})}{q_{1}} =\cdots= \frac{P_{n+1}(q_{1})}{q_{1}}=\nu, \qquad \frac{P_{2}(q_{1})}{q_{1}} = \vartheta, \qquad \frac{P_{1}(q_{1})}{q_{1}} =s,
\end{equation}
are satisfied and calculate back to derive the same
conditions \eqref{eq:hir} at $q_{0}$, using our choice of $\vartheta$.
It is clear that ultimately we can change the direction 
back to positive and repeat the period $[q_{0},q_{1}]$, 
blown up by the constant factor $q_{1}/q_{0}$ in each step, ad infinitum
again. 

Again first consider the special case $\nu=\nu_{0}=g_{n}(s)$. Let
$q_{0}, q_{1}$ be related via
\begin{equation}  \label{eq:tv2}
q_{1}= \frac{s-1}{\vartheta-1} q_{0}=
\frac{ \vartheta-1 }{ \nu-1 } q_{0}.
\end{equation}
The case $\nu=\vartheta=1$ is equivalent to $s=-n$ which we excluded.
We determine $\tilde{q}_{1}<q_{1}$ 
from intersecting the continuation of $P_{1}$ to the left 
decreasing
with slope $-n$ with the likewise continuation of
$P_{2}$ increasing with slope $+1$.
From equating $P_{1}$ and $P_{2}$
at $\tilde{q}_{1}$ we get
\[
P_{1}(\tilde{q}_{1})=sq_{1}+ n(q_{1}-\tilde{q}_{1}) = \vartheta q_{1} - (q_{1}-\tilde{q}_{1})=P_{2}(\tilde{q}_{1}).
\]
After some calculation and using \eqref{eq:hir}, \eqref{eq:tv}
we derive
\begin{equation} \label{eq:viech}
\tilde{q}_{1}= \frac{n+1+2s+(n-1)\nu}{n+1}\cdot q_{1},\quad
\frac{P_{1}(\tilde{q}_{1})}{\tilde{q}_{1}}= \frac{P_{2}(\tilde{q}_{1})}{\tilde{q}_{1}}= (1-n)\frac{s+n\nu }{n+1+2s+(n-1)\nu},
\end{equation}
and we recognize the right hand side as the value from \eqref{eq:tier}. When
moving to the left from $\tilde{q}_{1}$, we let $P_{1}, P_{2}$ exchange slopes 
at $\tilde{q}_{1}$ up to a point
$\tilde{q}_{2}< \tilde{q}_{1}$
where $P_{2}$ intersects $P_{3}=\cdots=P_{n+1}$ that rise with slope
$+1$ in $[\tilde{q}_{2},q_{1}]$.
From
\[
P_{3}(\tilde{q}_{2})=\cdots=P_{n+1}(\tilde{q}_{2})= \nu q_{1} - (q_{1}-\tilde{q}_{2})
=sq_{1}+ n(q_{1}-\tilde{q}_{2})= P_{2}(\tilde{q}_{2})
\]
we calculate 
\begin{equation}  \label{eq:viech4}
\tilde{q}_{2}= \frac{n+1+s-\nu}{n+1} q_{1}, \qquad \frac{P_{2}(\tilde{q}_{2})}{\tilde{q}_{2}}= \frac{n\nu+s}{n+1+s-\nu}.
\end{equation}
We further check by Proposition~\ref{qp}, and it follows from the construction below, that $\tilde{q}_{2}\geq q_{0}$.
At the switch point $\tilde{q}_{2}$, when going to the left
we change the slope of $P_{2}$ to $+1$ and the slopes of 
$P_{3},\ldots,P_{n+1}$ according to \eqref{eq:sumv} to $-2/(n-1)$,
recalling $P_{1}$ still has slope $+1$. Since
$P_{1}$ rises with slope $+1$ left of $\tilde{q}_{1}$, 
at some point $q^{\ast}<\tilde{q}_{1}$ we will
have $P_{1}(q^{\ast})=sq^{\ast}$. We will check that $q^{\ast}=q_{0}\leq \tilde{q}_{2}$ and that keeping the slopes
in $[q^{\ast}, \tilde{q}_{2}]$ equations \eqref{eq:hir} hold.
From equating
\[
P_{1}(q^{\ast})= sq^{\ast} =
\vartheta q_{1} - (q_{1}-q^{\ast})
\]
we indeed readily check that $q^{\ast}=q_{0}$ is the value as in \eqref{eq:tv2}.
Moreover, from \eqref{eq:tv2} we verify
\[
P_{2}(q_{0})= \nu q_{1} - (q_{1}-q_{0}) = \vartheta q_{0}.
\]
Since clearly $P_{3}(q_{0})=\cdots=P_{n+1}(q_{1})$ from
\eqref{eq:sumv} we conclude the remaining claims of \eqref{eq:hir}, 
proving our assertion.

Finally, for general $\nu\in [-s/n,\nu_{0}]$, we again split the interval
$[q_{0},\tilde{q}_{2}]$ into $[q_{0},\tilde{q}_{3}]$ and $[\tilde{q}_{3},\tilde{q}_{2}]$ for some 
$q_{0}\leq \tilde{q}_{3}\leq \tilde{q}_{2}$
and let $P_{2},\ldots,P_{n+1}$ all decay with slope $-1/n$ in 
$[\tilde{q}_{3},\tilde{q}_{2}]$, and in $[q_{0},\tilde{q}_{3}]$
we take the slopes $-2/(n-1)$ for $P_{3},\ldots,P_{n+1}$ 
and $+1$ for $P_{2}$, i.e. as in the interval $[q_{0},\tilde{q}_{2}]$ 
when $\nu=\nu_{0}$.
The value $\tilde{q}_{3}$ is again determined so
that the imposed assumptions \eqref{eq:hir} at $q_{0}$ are met. 
Similar to Theorem~\ref{thm1} we get
\[
\tilde{q}_{3} = \frac{1+s+(n-1)\nu}{n+1}\cdot \frac{1+n+(n-1)(s+n\nu)}{1-s} \cdot q_{1}. 
\]
We omit details of the calculation. This finishes the period and gives
rise to an $n$-template.

\begin{tikzpicture}

\draw[thick,->] (0,0) -- (14.2,0) node[anchor=west] {q};
\draw[thick,->] (0,-5) -- (0,2.5) node[anchor=south] {P(q)};


\draw[dotted,->] (0,0) -- (11,1.3) node[anchor=west] {$\overline{\varphi}_{3}=\cdots=\overline{\varphi}_{n+1}=\nu $};

\draw[dotted,->] (0,0) -- (11,-5) node[anchor=west] {$\underline{\varphi}_{1}=s $};

\draw[dotted,->] (0,0) -- (11,-1.5) node[anchor=west] {$\overline{\varphi}_{1}=\underline{\varphi}_{2}=\gamma $};

\draw[dotted,->] (0,0) -- (11,-.8) node[anchor=west] {$\vartheta $};

\draw[dotted,->] (0,0) -- (13.5,0.73) node[anchor=west] {$\overline{\varphi}_{2}$};

\draw[dotted,->] (0,0) -- (11,0.3) node[anchor=west] {$\underline{\varphi}_{3}=\cdots=\underline{\varphi}_{n+1}$};

\draw (9.8,-1.35) -- (10.5,-4.77) node[anchor=north] {$P_{1}$};

\draw (10.5,-0.75) -- (7.45,-3.4) node[anchor=north] {$P_{1}$};

\draw (8.66,0.45) -- (7.45,-0.55) node[anchor=north] {$P_{2}$};

\draw (9.8,-1.35) -- (9.45,0.3);


\draw (8.66,0.45) -- (7.45,0.88);

\draw [line width=0.4mm] (8.66,0.45) -- (7.45,0.88)  node[anchor=south east] {$P_{3}=\cdots=P_{n+1}$};

\draw [line width=0.4mm] (9.45,0.3) -- (10.5,1.25) node[anchor=south east] {$P_{3}=\cdots=P_{n+1}$};

\draw [line width=0.7mm] (9.45,0.3) -- (8.66,0.45);

\node at (10.6,-1.1) {$P_{2}$};

\node at (8.97,0.65) {\tiny$-\frac{1}{n}$};

\node at (8.18,0.85) {\tiny$-\frac{2}{n-1}$};

\node at (9.7,0.88) {\small$+1$};

\node at (8.1,-0.3) {\small$+1$};

\node at (8.5,-2.9) {$+1$};

\node at (10.5,-2.9) {$-n$};







\fill[black] (7.45,0) circle (0.06cm) node[anchor=north] {$q_{0}$};

\fill[black] (9.845,0) circle (0.06cm) node[anchor=north] {$\tilde{q}_{1}$};

\fill[black] (9.45,0) circle (0.06cm) node[anchor=north] {$\tilde{q}_{2}$};

\fill[black] (8.66,0) circle (0.06cm) node[anchor=north] {$\tilde{q}_{3}$};

\fill[black] (10.5,0) circle (0.06cm) node[anchor=north] {$q_{1}$};

\node at (5.5,-5.5) {Figure 4: Sketch period of $\textbf{P}_{s,\nu}$ };

\end{tikzpicture}

We again easily verify the claimed
upper and lower limits $\uv_{j}, \ov_{j}$ of the theorem.
Inserting $\uv_{1}=s, \overline{\varphi}_{n+1}=\nu$ and for $\overline{\varphi}_{1}$ from \eqref{eq:tier},
a calculation verifies equality in \eqref{eq:4}. 

Extending the interval for $\wo$ as in \eqref{eq:A2} works similarly
as in Section~\ref{ee} by
splitting the interval $[\tilde{q}_{2}, \tilde{q}_{1}]$ suitably
to attain given $\ov_{1}$ within a corresponding range,
we skip details.

To estimate the Hausdorff and packing dimensions in Theorem~\ref{thm2},
we evaluate the local contraction rates of 
$\textbf{P}= \textbf{P}_{s,\nu}$ within the period interval
$[q_{0},q_{1}]$ as
\[
\delta(\mathbf{P},q)=  \begin{cases}
n, & \text{if}\ q\in [q_{0}, \tilde{q}_{2}], \\
1, & \text{if}\ q\in [\tilde{q}_{2}, \tilde{q}_{1}], \\ 
0, & \text{if}\ q\in [\tilde{q}_{1}, q_{1}]. \\
\end{cases}
\]
By a similar argument as in Theorem~\ref{thm1} we see that to find the lower
limit we may consider the average contraction rate within the interval $[q_{0},q_{1}]$ and find
\begin{align*}
\dim_{H}(\Sigma_{s,\nu}^{\ast (n)}) &\geq \underline{\delta}(\mathbf{P})=\!\! \frac{\int_{q_{0}}^{q_{1}} \delta(\mathbf{P},q) \; dq}{q_{1} -q_{0}}
=\frac{n(\tilde{q}_{2}-q_{0})+ (\tilde{q}_{1}-\tilde{q}_{2})}{q_{1}-q_{0}}\\
&=
\frac{\tilde{q}_{1}+(n-1)\tilde{q}_{2}-nq_{0}}{q_{1}-q_{0}}= \frac{  \frac{\tilde{q}_{1}}{q_{1}}+(n-1)\frac{\tilde{q}_{2}}{q_{1}}-n\frac{q_{0}}{q_{1}}  }{ 1-\frac{q_{0}}{q_{1}} }.
\end{align*}
Inserting for the ratios $\tilde{q}_{1}/q_{1},\tilde{q}_{2}/q_{1},q_{0}/q_{1}$ from
\eqref{eq:tv2}, \eqref{eq:viech}, \eqref{eq:viech4} 
the bound becomes $D$ in \eqref{eq:D} after tedious rearrangements, 
verifying \eqref{eq:hd2}.

We finally estimate the packing dimension. 
Define $\tilde{q}_{1,N}, \tilde{q}_{2,N}$ within $[q_{N},q_{N+1}]$
corresponding to $\tilde{q}_{1}, \tilde{q}_{2}$ in $[q_{0},q_{1}]$ 
likewise as in the proof of Theorem~\ref{thm1}.
Then
\[
\dim_{P}(\Sigma_{s,\nu}^{\ast (n)}) \geq  
\overline{\delta}(\mathbf{P}) \geq \max\{ U,V\} 
\]
where $U$ resp. $V$ are the average limit contraction rates
in the intervals $[q_{0}, \tilde{q}_{2,N}]$ resp. $[q_{0}, \tilde{q}_{1,N}]$
as $N\to\infty$. We show $U\geq E, V\geq F$ with
$E,F$ from \eqref{eq:ee}. 
Since we identified $\underline{\delta}(\mathbf{P})$
as the average contraction rate in any interval $[q_{0},q_{N}]$ and $q_{0}=o(q_{N})$ as $N\to\infty$, we evaluate
\begin{align*}
U&= \lim_{N\to\infty} \frac{ \int_{q_{0}}^{\tilde{q}_{2,N}} \delta(\mathbf{P},q)\; dq }{ \tilde{q}_{2,N}-q_{0}  } =\lim_{N\to\infty} \frac{ \underline{\delta}(\mathbf{P})(q_{N}-q_{0}) + n(\tilde{q}_{2,N}-q_{N}) }{\tilde{q}_{2,N}-q_{0}} \\ &=
\lim_{N\to\infty} \frac{ \underline{\delta}(\mathbf{P})q_{N} + n(\tilde{q}_{2,N}-q_{N}) }{\tilde{q}_{2,N}} =
n + (\underline{\delta}(\mathbf{P})-n)\cdot \lim_{N\to\infty} \frac{q_{N}}{ \tilde{q}_{2,N} } 
\end{align*}
and since $q_{N}/\tilde{q}_{2,N}$
is independent of $N$, inserting $\underline{\delta}(\mathbf{P})\geq D$ this equals
\[
U\geq n + \frac{q_{0}}{ \tilde{q}_{2} } (\underline{\delta}(\mathbf{P})-n)
\geq n + (D-n)\frac{q_{0}}{ \tilde{q}_{2} } = n+ (D-n) \frac{(n+1)(s+(n-1)\nu+1)}{(1-s)(n+1+s-\nu)}= E,
\]
where we used \eqref{eq:tv2}, \eqref{eq:viech}, \eqref{eq:viech4} to 
evaluate $q_{0}/\tilde{q}_{2}$.
For $V$ a similar calculation shows
\begin{align*}
V&= \lim_{N\to\infty} \frac{ \int_{q_{0}}^{\tilde{q}_{1,N}} \delta(\mathbf{P},q)\; dq }{ \tilde{q}_{1,N}-q_{0}  } 
= \lim_{N\to\infty} \frac{ \int_{q_{0}}^{ q_{N+1}} \delta(\mathbf{P},q)\; dq - \int_{ \tilde{q}_{1,N} }^{ q_{N+1}} \delta(\mathbf{P},q)\; dq  }{ \tilde{q}_{1,N}-q_{0}  }\\ &=
\lim_{N\to\infty} \frac{(q_{N+1}-q_{0}) \underline{\delta}(\mathbf{P})}{\tilde{q}_{1,N}-q_{0}}= 
\underline{\delta}(\mathbf{P})\lim_{N\to\infty} \frac{q_{N+1} }{\tilde{q}_{1,N}}= \underline{\delta}(\mathbf{P})\frac{q_{1}}{ \tilde{q}_{1} }\\
&
\geq D \frac{q_{1}}{ \tilde{q}_{1} }= D\frac{n+1}{n+1+2s+(n-1)\nu}=F,
\end{align*}
as claimed, where we used \eqref{eq:viech} in the last line.
Theorem~\ref{thm2} is proved.

\subsection{Non-existence in Theorem~\ref{t3}}  \label{non2}

To complete the proof of Theorem~\ref{t3},
the following remains to be proved.

\begin{theorem}
	With the notation of Theorem~\ref{t3},
	for $\nu\notin [-s/n,\nu_{0}]$ the set $\Sigma_{s,\nu}^{(n)}$ is emtpy,
	i.e. there is no $\ux$ inducing
	equality in \eqref{eq:4} and with $\uv_{1}= s, \ov_{n+1}= \nu$.
\end{theorem}

For $\nu<-s/n$ again we get a contradiction to \eqref{eq:khin2}.
So it remains to exclude $\nu>\nu_{0}$. This can be done very similarly
as excluding $\mu<\mu_{0}$ in Theorem~\ref{t10} by some dual setup.
Again using the method from the proof
of~\cite[(7)]{ichnyj} one can show that for arbitrarily large $q$,
up to $o(q)$ as $q\to\infty$, we must
have a situation as in the interval $[\tilde{q}_{1},q_{1}]$ 
in Figure~4.  
Finally, for $\nu>\nu_{0}$ we again derive a contradiction when
considering the next meeting point of $L_{1}, L_{2}$ 
to the left of $\tilde{q}_{1}$.
We leave the details to the reader.
Again we can deduce equality in the dimension formulas for $s<0, \nu=\nu_{0}$
since then the entire combined graph must essentially be built
up from consecutive periodical patterns as in Figure~4.

\section{Final comments relating to work of Bugeaud, Laurent and Roy}  \label{comments}

According to the comments below~\cite[Theorem~3]{bulau}, sharpness
of \eqref{eq:1}, \eqref{eq:2} are respectively equivalent to
optimality of certain systems of inequalities.

Identify $\ux\in\mathbb{R}^n$ 
with its projective image in $P^n(\mathbb{R})$.
Let $0\leq d\leq n-1$ an integer. 
We denote by $\omega_{d}=\omega_{d}(\ux)$ the supremum of the real numbers
$u$ for which there exist infinitely many rational linear subvarieties $L\subseteq P^n(\mathbb{R})$ such that $\dim(L)=d$ and  $d(\ux,L)\leq H(L)^{-1-u}$, where $H(L)$ is the Weil height of any system of Pl\"ucker
coordinates of $L$ and $d(A,B)$ denotes the distance
of two projective sets $A,B\subseteq P^n(\mathbb{R})$. 
Then $\omega_{0}$ corresponds
to the classical exponent $\om$, and $\omega_{n-1}$ to our $\os$.
Translating~\cite[Proposition~3.1]{royspec} to our formalism, 
any $\omega_{d}$ can be written as an expression 
involving certain $\varphi_{j}(q)$ via
\begin{equation}  \label{eq:omd}
\frac{1}{1+\omega_{d}}= \limsup_{q\to \infty} \frac{n-d-
	\sum_{j=d+2}^{n+1} \varphi_{j}(q)}{n+1}, \qquad\qquad 0\leq d\leq n-1,
\end{equation}
upon the convention $\omega_{d}=\infty$ if the right hand side becomes $0$. 

Analyzing the proof of \eqref{eq:2} in~\cite{bulau}, 
identity is equivalent to identities
\begin{equation}  \label{eq:vi}
\omega_{1}=\frac{\om+ \wo}{1-\wo}
\end{equation}
and
\begin{equation}  \label{eq:iii}
\omega_{d+1}=\frac{(n-d)\omega_{d}+1}{n-d-1}, 
\qquad\qquad 1\leq d\leq n-2.
\end{equation}
For general $\ux\in\mathbb{R}^n$, 
the left hand sides are bounded below by the right hand sides 
in \eqref{eq:vi}, \eqref{eq:iii} for every $0\leq d\leq n-2$, so we have inequalities.
A special case of a result by Roy~\cite[Theorem~2.3]{royspec} implies
that for any reasonable choice of $\omega_{0}=\om$, 
there exists $\ux\in\mathbb{R}^n$ that simultaneously satisfies
all identities in \eqref{eq:iii} (for $d=0$ as well). 
In fact, Marnat~\cite{marnat} evaluated 
Hausdorff and packing dimensions of the corresponding 
sets of $\ux\in\mathbb{R}^n$ with the aid of Theorem~\ref{deep}. 
However, the condition \eqref{eq:vi} 
on $\wo$ remains. Indeed, the restricitons 
from Theorem~\ref{t2}
show that we cannot have this additional identity in certain cases.
We mention that if \eqref{eq:vi} and an extension
of \eqref{eq:iii} also valid for $d=0$ hold, then we could conclude
$\wo=1/n$. In particular, if $\om>1/n$ then
our vectors $\ux$ in Theorem~\ref{t2} do not have these properties
as it can be checked that they induce $\wo>1/n$. 
Similarly, looking at the proof of \eqref{eq:1} in~\cite{bulau} 
we check that
equality happens if and only if
\begin{equation}  \label{eq:hiatus}
\omega_{d-1}=\frac{d\omega_{d}+1}{\omega_{d}+d-1}, 
\qquad\qquad 1\leq d\leq n-2,
\end{equation}
and
\begin{equation}  \label{eq:nn1}
\omega_{n-2} = \frac{ (\wos-1)\os }{ \os + \wos }.
\end{equation}
Again in general there are just inequalities in \eqref{eq:hiatus}
and \eqref{eq:nn1}, with the right hand sides not exceeding the left.
Again \eqref{eq:hiatus} can be satisfied for $\ux$
by Roy's~\cite[Theorem~2.3]{royspec}, however 
we are still left with a condition on $\wos$. Again, 
if we impose \eqref{eq:nn1} and an extension of \eqref{eq:hiatus} 
valid for $d=n-1$ as well, we conclude
$\wos=n$, so for $\os>n$ the examples in Theorem~\ref{t0}
do not satisfy these properties.


\end{document}